\documentclass[12pt, reqno]{amsart}
\usepackage{amssymb}
\usepackage{eucal}
\usepackage{amsmath}
\usepackage{amscd}
\usepackage[dvips]{color}
\usepackage{multicol}
\usepackage[all]{xy}           
\usepackage{graphicx}
\usepackage{color}
\usepackage{colordvi}
\usepackage{xspace}
\usepackage{txfonts}
\usepackage{lscape}
\usepackage{tikz}

\usepackage[shortlabels]{enumitem}
\usepackage{ifpdf}
\ifpdf
  \usepackage[colorlinks,final,backref=page,hyperindex]{hyperref}
\else
  \usepackage[colorlinks,final,backref=page,hyperindex,hypertex]{hyperref}
\fi
\usepackage{tikz}

\numberwithin{equation}{section}

\def\d{{\rm d}}
\def\al{\alpha}
\def\be{\beta}

\def\Der{{\rm Der}}
\def\Inn{{\rm Inn}}

\def\W{\mathcal{W}}
\def\Z{\mathbb{Z}}
\def\C{\mathbb{C}}

\def\ot{\otimes}

\def\de{\delta}

\def\bN{{\mathbb N}}
\def\bZ{{\mathbb Z}}

\def\bC{{\mathbb C}}

\numberwithin{equation}{section}
\newtheorem{theo}{Theorem}[section]
\newtheorem{defi}[theo]{Definition}
\newtheorem{coro}[theo]{Corollary}
\newtheorem{lemm}[theo]{Lemma}
\newtheorem{prop}[theo]{Proposition}

\newtheorem{rema}[theo]{Remark}

\def\Der{\mbox{\rm Der}}
\def\Inn{\mbox{\rm Inn}}

\title[1-cocycles of the Witt algebra]{1-cocycles of the Witt algebra with coefficients in tensor product of modules}

\author[Gao]{Shoulan Gao}
\address{Department of Mathematics, Huzhou University, Zhejiang Huzhou
313000, PR China}\email{gaoshoulan@zjhu.edu.cn}

\author[Liu]{Dong Liu}
\address{Department of Mathematics, Huzhou University, Zhejiang Huzhou
313000, PR China}
\email{liudong@zjhu.edu.cn}

\author[Pei]{Yufeng Pei$^*$}
\address{Department of Mathematics, Huzhou University, Zhejiang Huzhou
313000, PR China}\email{pei@zjhu.edu.cn}

\date{}


\thanks{$*$: Corresponding author.}

\begin{document}
\maketitle

\begin{abstract}
In this paper, we classify 1-cocycles of the Witt algebra with coefficients in the tensor product of two arbitrary tensor density modules.
In a special case, we recover a theorem originally established by Ng and Taft in \cite{NT}. Furthermore, by  these 1-cocycles, we determine Lie bialgebra structures over certain infinite-dimensional Lie algebras containing the Witt algebra.

\end{abstract}


\setcounter{section}{0}

\allowdisplaybreaks

\section{Introduction}

The Witt algebra $\mathcal{W}$ is an infinite-dimensional Lie algebra that emerges as the algebra of derivations acting on the Laurent polynomial ring in a single variable. The Virasoro algebra, which is the universal central extension of the Witt algebra, holds significant importance in the realms of both mathematics and theoretical physics.   The cohomology of the Witt algebra and the Virasoro algebra has been an active area of research owing to its deep relations with representation theory \cite{FF,IK,MS,MP}. In particular, the algebraic cohomology of the Witt algebra, when considering coefficients related to the tensor density modules $\mathcal{F}_{\alpha}$, has been the subject of study in various works  \cite{E,ES,ES1,Fi,GR,OR2,Sc}. In \cite{FF}, Feigin and Fuchs achieved a highly non-trivial classification of length two extensions for the Witt algebra  using the first algebraic cohomology, with
coefficients in the module of the space of homomorphisms between the two arbitrary tensor density modules.   Recently, in \cite{D}, Dykes computed all possible 1-cocycles for the Witt algebra by employing a compatible action of the commutative algebra of Laurent polynomials in one variable.

  Motivated by Witten's study of left-invariant symplectic structures on the Virasoro group (\cite{Wi}), Ng and Taft established that all Lie bialgebra structures  on the Witt algebra and the Virasoro algebra are triangular coboundary through computation of 1-cocycles with  coefficients in the tensor product of the adjoint modules in \cite{NT} (see also \cite{KS}).
The Lie bialgebra structures on various infinite-dimensional Lie algebras related to the Witt algebra have been investigated by computing the first algebraic cohomology \cite{FLL,HLS,LPZ,LSX,SS}.   As mentioned in \cite{LPZ}, a key and complicated step of characterizing these Lie bialgebras is dealing with 1-cocycles of the Witt algebra with coefficients in the tensor product of  two tensor density modules. To the best knowledge of the authors, there is no general formula for these 1-cocycles in the literature.

This paper aims to classify 1-cocycles of the Witt algebra  with coefficients in  the tensor product of  two arbitrary tensor density modules.  Our main theorem can be stated as follows:

\vskip 5pt

\noindent\textbf{Main Theorem} (Theorem \ref{T}): For $\alpha, \beta \in \mathbb{C}$, we have

\begin{equation*}
\dim \mathrm{H}^1(\mathcal{W}, \mathcal{F}_{\alpha} \otimes \mathcal{F}_{\beta}) = \begin{cases}
5, & (\alpha, \beta) = (0,0); \\
2, & (\alpha, \beta) = (0,1) \text{ or } (1,0); \\
1, & (\alpha, \beta) = (0,2), (2,0), (1,1), \text{ or } \alpha + \beta = 1, \beta \neq 0,1; \\
0, & \text{otherwise}.
\end{cases}
\end{equation*}

Establishing the proof for this theorem is challenging and involves numerous calculations. Nonetheless, the fundamental concept underlying the proof remains clear. First, suitable normalizations allow selecting well-defined representatives for each cohomology class, reducing the computation of $\mathrm{H}^1(\mathcal{W}, \mathcal{F}_{\alpha} \otimes \mathcal{F}_{\beta})$ to that of normalized 1-cocycles. Through several technical lemmas, we then demonstrate a normalized 1-cocycle must assume a  specialized form, enabling an upper bound on $\dim \mathrm{H}^1(\mathcal{W}, \mathcal{F}_{\alpha} \otimes \mathcal{F}_{\beta})$. Finally, exhibiting linearly independent normalized 1-cocycles produces the result.

The interest of this classification for these 1-cocycles lies in the applications to the Lie bialgebras associated with some infinite dimensional Lie algebras.
When $(\alpha,\beta)=(-1,-1)$, we recover a theorem first proven by Ng and Taft in \cite{NT}. Furthermore, as consequence of the Main Theorem, we establish Lie bialgebra structures over various infinite-dimensional Lie algebras that contain the Witt algebra, such as the Oveshenko-Roger algebra \cite{OR2}, the twisted Heisenberg-Virasoro algebra \cite{ADKP}, the BMS algebra \cite{BGMM, ZD}, and the Schrödinger-Virasoro algebra \cite{RU}.

It is worth noting that, as stated in \cite{FF}, an isomorphism exists:
$
\mathrm{H}^1(\mathcal{W},\mathrm{Hom}(\mathcal{F}_{\alpha},\mathcal{F}_{\beta}))\cong \mathrm{H}_1(\mathcal{W}, \mathcal{F}_{\alpha} \otimes \mathcal{F}_{\beta}'),$
where $\mathcal{F}_{\beta}'$ denotes the graded dual modules of $\mathcal{F}_{\beta}$.
Our main theorem may be considered as a dual counterpart to the 1-cycles presented in \cite{FF}. However, there is currently no more efficient method for computing 1-cocycles via 1-cycles.

The paper is organized as follows. In Section \ref{section 2}, we introduce the notations and conventions used throughout the paper, we also present several technical lemmas on 1-cocycles of degree zero that will be useful later. In Section \ref{section 3}, we construct several 1-cocycles of the Witt algebra.
Section \ref{section 4} contains the proof of Theorem \ref{T} for the case $(\alpha, \beta) \notin \mathbb{Z}^2$, while Section \ref{section 5} handles the case $(\alpha, \beta) \in \mathbb{Z}^2$. Finally, in Section \ref{section 6} we apply Theorem \ref{T} to investigate several infinite-dimensional Lie bialgebras containing $\mathcal{W}$.

Throughout the paper, we use $\mathbb{C}, \mathbb{C}^*, \mathbb{Z}, \mathbb{Z}_{+}, \mathbb{N}$, and $\mathbb{Z}^*$ to denote the sets of complex numbers, non-zero complex numbers, integers, positive integers, non-negative integers, and non-zero integers, respectively.

\section{Preliminaries and notations}\label{section 2}

In this section, we review some preliminaries and notations  on 1-cocycles of the Witt algebra and obtain several useful lemmas.

\begin{defi}
The Witt algebra $\W$ is an infinite-dimensional  Lie algebra spanned by $\{L_n\mid n\in \bZ\}$ with the bracket
$$
[L_m, L_n]=(m-n)L_{m+n},\quad \forall m,n\in\bZ.
$$
\end{defi}

\begin{defi}
 For $\al\in \bC$, the tensor density module $\mathcal{F}_{\al}=\bigoplus_{n\in\bZ}\bC v_n$ for  $\W$ is given by
$$
L_m \cdot v_n=-(\al m+n)v_{m+n},\quad  \forall m,n\in\bZ.
$$
\end{defi}

 For $\al,\be\in\bC$,  $\mathcal{F}_{\al}\ot \mathcal{F}_{\be}$ is a $
\W$-module with
$$
L_{m}\cdot (v_{i}\ot v_{j})=-(i+\al m)v_{m+i}\ot v_{j}-(j+\be m)v_{i}\ot v_{m+j}, \quad\forall m, i, j\in\bZ.
$$
In particular,
$$
L_{0}\cdot (v_{i}\ot v_{j})=-(i+j)v_{i}\ot v_{j},\quad \forall i,j\in\bZ.$$  It follows that $\mathcal{F}_{\al}\ot \mathcal{F}_{\be}$  is equipped with a $\bZ$-grading:
$$
\mathcal{F}_{\al}\ot \mathcal{F}_{\be}=\bigoplus\limits_{k\in\bZ}(\mathcal{F}_{\al}\ot \mathcal{F}_{\be})_k,\quad\text{where}\quad (\mathcal{F}_{\al}\ot \mathcal{F}_{\be})_k=\bigoplus\limits_{i\in\bZ}\bC v_{i}\ot v_{k-i}.
$$

\begin{defi}[\cite{We}]
    Let $L$ be a Lie algebra and  $M$  an $L$-module. Denote by $\Der(L, M)$ the set of 1-cocycles (derivations) $d:L\rightarrow M$ with coefficients in $M$, which are linear maps satisfying the condition $d([x,y])=x\cdot d(y)-y\cdot d(x)$ for $x,y\in L$. The set $\Inn(L, M)$ consists of the 1-cocycles $d_v$ for some $v\in M$, where $d_v$ is called  { 1-coboundary (inner derivation) } with coefficients in $M$ defined by $d_{v}:x\mapsto x\cdot v$ for $x\in L$.  A  1-cocycle with coefficients in $M$  is called non-trivial if it is not a  1-coboundary with coefficients in $M$. We denote by ${\rm H}^1(L, M)=\Der(L, M)/\Inn(L, M)$ the first algebraic cohomology of $L$ with coefficients in $M$.
\end{defi}

\begin{rema}
    {\rm  It is a well-known fact \cite{Fu} that there exists a one-to-one correspondence between elements in ${\rm H}^1(L, M)$ and equivalence classes of extensions of $L$-modules by $M$. }
\end{rema}

\begin{defi}
    Let $G$ be an abelian group, $L=\bigoplus_{g \in G} L_g$ a $G$-graded Lie algebra over  $\bC$.  An $L$-module $V$ is said to be $G$-graded if $V=\bigoplus_{g \in G} V_g$ and $L_g \cdot V_h \subset V_{g+h}$ for all $g, h \in G$.
\end{defi}

\begin{prop}[\cite{Fa}]\label{Fa}
 Assume that $L$ is finitely generated as a $G$-graded Lie algebra. Let $V$ be a G-graded $L$-module. Then
$$
\operatorname{Der}(L, V)=\underset{g \in G}{\bigoplus} \operatorname{Der}(L, V)_g.
$$

\end{prop}

\begin{lemm}\label{L1.1}The derivation algebra $\Der(\W,\mathcal{F}_{\al}\ot \mathcal{F}_{\be})$ is $\bZ$-graded, i.e.,
\begin{eqnarray*}
&&\Der(\W,\mathcal{F}_{\al}\ot \mathcal{F}_{\be})=\bigoplus\limits_{k\in\bZ}\Der(\W,\mathcal{F}_{\al}\ot \mathcal{F}_{\be})_{k},\quad D_k(L_{m})\subseteq (\mathcal{F}_{\al}\ot \mathcal{F}_{\be})_{m+k}
\end{eqnarray*}
for any $m\in\bZ, D_k\in\Der(\W,\mathcal{F}_{\al}\ot \mathcal{F}_{\be})_{k}$.
\end{lemm}
\begin{proof}
According to Proposition \ref{Fa}, coupled with the knowledge that the Witt algebra $\W$ is both finitely generated and $\bZ$-graded.
\end{proof}

\begin{lemm}\label{L1.2}
$\Der(\W,\mathcal{F}_{\al}\ot \mathcal{F}_{\be})=\Der(\W,\mathcal{F}_{\al}\ot \mathcal{F}_{\be})_{0}+\Inn(\W,\mathcal{F}_{\al}\ot \mathcal{F}_{\be})$.
\end{lemm}
\begin{proof}

Assuming $k\neq 0$ and $D_k\in \Der(\W,\mathcal{F}_{\al}\ot \mathcal{F}_{\be})_{k}$, we have $D_{k}(L_{0})=\sum\limits_{r=0}^{s} a(i_{r})v_{i_{r}}\ot v_{k-i_{r}}$, where $a(i_r)\in\bC$ and $s\in\bN$. Using the equations $D_k[L_n,L_0]=L_n\cdot D_k(L_0)-L_0 \cdot D_k(L_n)$ and $D_k(L_n)\subseteq (\mathcal{F}_{\al}\ot \mathcal{F}_{\be})_{k+n}$, we obtain
$L_n\cdot D_k(L_0)=-k D_k(L_n)$
for any integers $n,k$. Therefore,
$$D_k(L_n)= - \frac1{k} L_n\cdot D_k (L_0)= \frac1{k}\sum _{r=1}^sa(i_r)\Big((i_r+\alpha n)v _{n+i_r}\otimes v _{k-i_r}+( k-i_r+\beta n)v_i{_ri }\otimes v _ {n+k-i_r}\Big).$$
This implies that
$$D_k (L_n)= L_n \cdot  \left(\sum_ { r = 1 }^s{\frac{a(i_r)}{k}}v_{ i_r}\otimes v_{ k - i_r }\right), $$
which shows that $D_k$ is a 1-coboundary.

\end{proof}

\begin{lemm}\label{L1.3}
For any $D\in\Der(\W,\mathcal{F}_{\al}\ot \mathcal{F}_{\be})_{0}$, we have
$
D(L_{0})=a \delta_{\al,0}\delta_{\be,0} v_{0}\otimes v_{0}
$
for some $a\in\bC$.
\end{lemm}
\begin{proof}
Let $D\in\Der(\W,\mathcal{F}_{\al}\ot \mathcal{F}_{\be})_{0}$ and suppose that for any $n\in\bZ^{*}$, we have $D(L_{n}) = \sum_{i\in\Z} a_{n,i}v_i \otimes v_{n-i}$, where $a_{n, i}\in\C$. It should be noted that $L_n \cdot D(L_0)=0$ for all $n\in\bZ$. Additionally, assume that $D(L_0)=\sum_{i=p}^{q} a(i)v_i \otimes v_{-i}$, where $a(i)\in\bC$ and $p, q$ are integers.
Then
\begin{equation}\nonumber
L_{n}\cdot D(L_{0})=-\sum\limits_{i=p}^{q} a(i)\Big((i+\al n)v_{n+i}\ot v_{-i}+(-i+\be n)v_{i}\ot v_{n-i}\Big), \quad\forall  n\in\bZ.
\end{equation}
Taking $n=1,-1$ respectively, it follows that
\begin{eqnarray*}
\sum\limits_{i=p}^{q} a(i)\Big((i+\al )v_{1+i}\ot v_{-i}+(-i+\be )v_{i}\ot v_{1-i}\Big)=\sum\limits_{i=p}^{q} a(i)\Big((i-\al )v_{i-1}\ot v_{-i}+(-i-\be )v_{i}\ot v_{-i-1}\Big)=0.
\end{eqnarray*}
Then
\begin{equation}\label{1.4a}
 a(p)(-p+\be )=0,\quad a(p)(p-\al )=0,\quad a(q)(q+\be )=0,\quad  a(q)(q+\al )=0;
\end{equation}
\begin{equation}\label{1.4aa}
a(i)(-i+\be )+a(i-1)(i-1+\al)=0,\quad\quad\, p+1\leq i\leq q;
\end{equation}
\begin{equation}\label{1.5aa}
a(i+1)(i+1-\al)-a(i)(i+\be )=0,\quad\quad\, p\leq i\leq q-1.
\end{equation}
It is clear that
$
a(p)(\al-\be)=a(q)(\al-\be)=0.
$
Replacing $i$ by $i-1$ in (\ref{1.5aa}), we have
\begin{equation}\label{1.5aab}
a(i)(i-\al)-a(i-1)(i-1+\be )=0, \qquad\, p+1\leq i\leq q.
\end{equation}
Using (\ref{1.4aa}) and (\ref{1.5aab}), we have
$[a(i)-a(i-1)](\al-\be)=0$ for $p+1\leq i\leq q.$

{\bf Case 1: $\alpha\neq \beta$.} Clearly, $a(i)=0$ for all $p\leq i\leq q$, which implies that $D(L_{0})=0$.

{\bf Case 2: $\alpha=\beta\notin\bZ$.} For any integers $p,q$, we have $a(p)=a(q)=0$. Therefore, using (\ref{1.4aa}) or (\ref{1.5aa}), we conclude that $a(i)=0$ for all $p\leq i\leq q$, and hence, $D(L_{0})=0$.

{\bf Case 3: $\alpha=\beta\in\bZ$.} Without loss of generality, let us assume that both $a(p)$ and $a(q)$ are non-zero. Then from (\ref{1.4a}), it follows that $\beta=\alpha=p=-q$.

If $\alpha=\beta\neq 0$, then we must have $p<0$, and also, since the coefficients vanish outside the range of indices $(p,q)$, we get that $q>0$. In this case,
$$
D(L_{0})=\sum_{i=-q}^{q} a(i) v_i \otimes v_{-i}.
$$
Now applying the operator ${L_n}$ to this expression yields
$$
-\sum_{i=-q}^{q}\Big(a(i)(i+\alpha n)v_{n+i}\otimes v_{-i}- a(i)(-i+\beta n)v_i \otimes v_{n-i}\Big)= 0,\quad \forall n \in\bZ.
$$
Thus, for every positive integer value of~$n,$ we obtain two equations:
$a(q)(n-q-\alpha n)=0$ and $a(-q)(n+q-\beta n)=0.$
Since both $a(q)$ and $a(-q)$ are non-zero, we must have that $\alpha=\beta=0$, which leads to a contradiction.

Therefore, we conclude that $\alpha=\beta=0$. The proof is now complete.

\end{proof}

Let $D\in\Der(\W, \mathcal{F}_{\al}\ot \mathcal{F}_{\be})_{0}$ and
\begin{eqnarray*}
D(L_{n}) =\sum\limits_{i\in\Z} a_{n,i}v_i\!\otimes\!v_{n-i}, \quad\forall\,n\in\bZ^{*},
\end{eqnarray*}
where $a_{n, i}\in\C$. We denote by
$$l=\max\{\,|i\,|\,\big|\,a_{1,i}\ne0\}.$$

\begin{lemm}\label{L1.4}If $\be\notin\bZ$, there exist  $c_{i}\in\bC(i=-l,\cdots\, l)$ such that
$$D(L_1)=L_1\cdot\left( \sum\limits_{i=-l}^{l}c_{i}v_{i}\otimes\,v_{-i}\right)+x_{l}(\al+l)v_{l+1}\otimes\,v_{-l}.$$

\end{lemm}
\begin{proof}
Suppose $\be\notin\bZ$. For any $l\in\mathbb{Z}_{+}$, we have $D(L_1)=\sum\limits_{i=-l}^{l}a_{1,i}v_{i}\otimes\,v_{-i+1}$ and
\begin{eqnarray*}
&&L_{1}\cdot\left(\sum\limits_{i=-l}^{l}x_{i}v_{i}\otimes\,v_{-i}\right)=-\sum\limits_{i=-l}^{l}x_{i}\Big((i+\al)v_{i+1}\otimes\,v_{-i}+(-i+\be)v_{i}\otimes\,v_{-i+1}\Big)\\
&=&-x_{-l}(l+\be)v_{-l}\otimes\,v_{l+1}-\sum\limits_{i=-l+1}^{l}\Big((i-1+\al)x_{i-1}+(\be-i)x_{i}\Big)v_{i}\otimes\,v_{-i+1}-x_{l}(l+\al)v_{l+1}\otimes\,v_{-l}.
\end{eqnarray*}
Since $\be\notin\bZ$,  we have $\be+i\neq 0$ for $-l\leq i\leq l$. So the following system of linear equations admits a unique solution:
\begin{eqnarray*}
& &x_{-l}(l+\be)=-a_{1,-l}, \quad (i-1+\al)x_{i-1}+(\be-i)x_{i}=-a_{1,i}, \quad\quad\, -l+1\leq i\leq l.
\end{eqnarray*}
Let $x_{i}=c_{i}(i=-l, \cdots l)$ be the solution. The proof is complete.

\end{proof}

\begin{lemm}\label{L1.4e1}

{\rm (1)}\ If $\be\in\bZ$ and $|\be|>l$, there exist some $x_{i}\in\bC$ such that
$$D(L_1)=L_1\cdot\left( \sum\limits_{i=-l}^{l}x_{i}v_{i}\otimes\,v_{-i}\right)+x_{l}(\al+l)v_{l+1}\otimes\,v_{-l}.$$

{\rm (2)}\ If $\al\in\bZ$ and $\al<-l+1$ or {$\al>l+1$},  there exist  $y_{i}\in\bC$ such that
$$D(L_1)=L_{1}\cdot\left( \sum\limits_{i=-l-1}^{l-1}y_{i}v_{i}\otimes\,v_{-i}\right)+y_{-l-1}(\be+l+1)v_{-l-1}\otimes\,v_{l+2}.$$
\end{lemm}
\begin{proof}
Similar to the proof of Lemma {\ref{L1.4}}.
\end{proof}

\begin{lemm}\label{L1.4e4}

{\rm (1)}\ If $\be\in\bZ$, and $-l\leq \be \leq l$, there exist some $x_{i}\in\bC(i=-l,\cdots\, l)$ such that
$$D(L_1)=\Big(a_{1,\be}+{(\al+\be-1)}x_{\be-1}\Big)v_{\be}\otimes\,v_{-\be+1} + L_1\cdot\left( \sum\limits_{i=-l}^{l}x_{i}v_{i}\otimes\,v_{-i}\right)+x_{l}(\al+l)v_{l+1}\otimes\,v_{-l},$$
where $x_{\be}=0$.\par

{\rm (2)}\ {If $\al\in\bZ$, and $-l+2\leq\al\leq l+1$, there exist some $y_{i}\in\bC(i=-l-1,\cdots\, l-1)$ such that
$$D(L_1)=\Big(a_{1,-\al}+(\al+\be-1)y_{-\al+1}\Big)v_{-\al+1}\otimes\,v_{\al} + L_{1}\cdot\left( \sum\limits_{i=-l-1}^{l-1}y_{i}v_{i}\otimes\,v_{-i}\right)+y_{-l-1}(\be+l+1)v_{-l-1}\otimes\,v_{l+2},$$
where $y_{-\al}=0$}.
\end{lemm}
\begin{proof}
Suppose $\be\in\bZ$, and $-l\leq \be \leq l$. Then $D(L_1)=\sum\limits_{i=-l}^{l}a_{1,i}v_{i}\otimes\,v_{-i+1}$ and
\begin{eqnarray*}
& &L_{1}\cdot\left(\sum\limits_{i=-l}^{l}x_{i}v_{i}\otimes\,v_{-i}\right)  =-\sum\limits_{i=-l}^{l}x_{i}\Big((i+\al)v_{i+1}\otimes\,v_{-i}+(-i+\be)v_{i}\otimes\,v_{-i+1}\Big)\\
&=&-x_{-l}(l+\be)v_{-l}\otimes\,v_{l+1}-\sum\limits_{i=-l+1}^{l}\Big((i-1+\al)x_{i-1}+(\be-i)x_{i}\Big)v_{i}\otimes\,v_{-i+1}-x_{l}(l+\al)v_{l+1}\otimes\,v_{-l},
\end{eqnarray*}
where $x_{\be}=0$. Then the following system of linear equations admits a unique solution:
\begin{eqnarray*}
& &x_{-l}(l+\be)=-a_{1,-l},\quad (i-1+\al)x_{i-1}+(\be-i)x_{i}=-a_{1,i}, \quad -l+1\leq i\leq l, i\neq\be.
\end{eqnarray*}
Furthermore,
$$D(L_1)=\Big(a_{1,\be}+{(\al+\be-1)}x_{\be-1}\Big)v_{\be}\otimes\,v_{-\be+1} + L_1\cdot\left( {\sum\limits_{i=-l}^{l}}x_{i}v_{i}\otimes\,v_{-i}\right)+x_{l}(\al+l)v_{l+1}\otimes\,v_{-l}.$$

The proof of (2) is similar.
\end{proof}

\begin{lemm}\label{L1.5e}
If $D(L_{0})=0$, we have
\begin{equation}\label{1.6ccc}
 (i+m-m\al)a_{m,i+m}-(i-m+m\be)a_{m,i}=(i-m+m\al)a_{-m,i-m}-(i+m-m\be)a_{-m,i}
\end{equation}
for all $m,i\in\bZ$.
\end{lemm}
\begin{proof}
If $D(L_{0})=0$, applying $D$ on $[L_{m}, L_{-m}]=2mL_{0}$, we have
$
L_{m}\cdot D(L_{-m})=L_{-m}\cdot D(L_{m})
$
for $m\in\bZ$. It follows that
\begin{equation*}
 (i+m-m\al)a_{m,i+m}-(i-m+m\be)a_{m,i}=(i-m+m\al)a_{-m,i-m}-(i+m-m\be)a_{-m,i}, \quad\forall m, i\in\bZ.
\end{equation*}
\end{proof}

\begin{lemm}\label{L1.5ee}
For all $i\in\bZ$, we have
\begin{equation}\label{1.6c}
 (i-2+2\al)a_{-1,i-2}-(i+1-2\be)a_{-1,i}+3a_{1,i}=(i+1-\al)a_{2,i+1}-(i-2+\be)a_{2,i},
\end{equation}
\begin{equation}\label{1.6cc}
 (i+2-2\al)a_{1,i+2}-(i-1+2\be)a_{1,i}-3a_{-1,i}=(i-1+\al)a_{-2,i-1}-(i+2-\be)a_{-2,i}.
\end{equation}
\end{lemm}
\begin{proof}
Applying $D$ on  $[L_{-1}, L_{2}]=-3L_{1}$, we have
$$L_{2}\cdot\left(\sum a_{-1,i}v_{i}\otimes v_{-i-1}\right)-3\sum a_{1,j}v_{j}\otimes v_{-j+1}=L_{-1}\cdot\left(\sum a_{2,k}v_{k}\otimes v_{-k+2}\right).$$
Therefore
\begin{equation*}
 (i-2+2\al)a_{-1,i-2}-(i+1-2\be)a_{-1,i}+3a_{1,i}=(i+1-\al)a_{2,i+1}-(i-2+\be)a_{2,i}, \quad\, i\in\bZ.
\end{equation*}

Applying $D$ on  $[L_{1}, L_{-2}]=3L_{-1}$, we have
$$L_{-2}\cdot\left(\sum a_{1,i}v_{i}\otimes v_{-i+1}\right)+3\sum a_{-1,j}v_{j}\otimes v_{-j-1}=L_{1}\cdot\left(\sum a_{-2,k}v_{k}\otimes v_{-k-2}\right).$$
It follows that
\begin{equation*}
 (i+2-2\al)a_{1,i+2}-(i-1+2\be)a_{1,i}-3a_{-1,i}=(i-1+\al)a_{-2,i-1}-(i+2-\be)a_{-2,i}, \quad\, i\in\bZ.
\end{equation*}
\end{proof}

\section{Construction of 1-cocycles  with coefficients in $\mathcal{F}_{\al}\ot \mathcal{F}_{\be}$}\label{section 3}

In this section, we shall construct several non-trivial
1-cocycles of $\mathcal{W}$ with coefficients in $\mathcal{F}_{\al}\ot \mathcal{F}_{\be}$.

\begin{prop}\label{P1.6--}
For any $\be\in\C$, we define a linear operator
 $\d_{1-\be, \be}: \W\to \mathcal{F}_{1-\be}\ot \mathcal{F}_{\be}$ as follows:
$$\d_{1-\be, \be}(L_{n})=
\begin{cases}
\sum\limits_{i=0}^{n-1}(i-n\be)v_{i}\otimes v_{-i+n},&n\geq 1;\\
\qquad\, 0,&n=0;\\
-\sum\limits_{i=n}^{-1}(i-n\be)v_{i}\otimes v_{-i+n},&n\leq -1.
\end{cases}
$$
Then $\d_{1-\be, \be}$ is a non-trivial 1-cocycle of $\W$ with coefficients in $\mathcal{F}_{1-\be}\ot \mathcal{F}_{\be}$.

\end{prop}
\begin{proof}

It is straightforward to check that $\d_{1-\be, \be}\in \Der(\W,\mathcal{F}_{1-\be}\ot \mathcal{F}_{\be})$ . Assume $\d_{1-\be, \be}$ is a 1-coboundary. Then
$$\d_{1-\be, \be}(L_{n})=L_{n}\cdot\left(\sum\limits_{i=p}^{q}a_{i}v_{i}\otimes v_{-i}\right)=\sum\limits_{i=p}^{q}a_{i}\Big((i-n\be)v_{i}\otimes v_{-i+n}-(i+n-n\be)v_{i+n}\otimes v_{-i}\Big)$$
for some $a_{i}\in\bC$, where $a_{p}, a_{q}\in\bC^{*}$ and $p\leq q$. Since $\d_{1-\be, \be}(L_{1})=-\be v_{0}\otimes v_{1}$, we have
$$-\be v_{0}\otimes v_{1}=(p-\be)a_{p}v_{p}\otimes v_{-p+1}+\sum\limits_{i=p+1}^{q}(i-\be)[a_{i}-a_{i-1}]v_{i}\otimes v_{-i+1}-(q+1-\be)a_{q}v_{q+1}\otimes v_{-q}.$$

{\bf Case 1: $\be\notin\bZ$.} If $p\neq 0$, then $(p-\be)a_{p}=0$ implies that $a_p=0$ or $p=\be$. However, since we know that $a_p\neq 0$, it follows that $p=\be$, which is impossible. Similarly, if $p=0$ and $q>0$, then $(q+1-\be)a_{q}=0$, which is also impossible. Therefore, we must have $p=q=0$. In this case,
$$
-\beta v_0 \otimes v_1 = (-\beta a_0) v_0 \otimes v_1 - (1-\beta a_0) v_1 \otimes v_0,
$$
which is again impossible. Hence, $\d_{1-\be,\be}$ is a non-trivial 1-cocycle.

{\bf Case 2:\ $ \be\in\bZ$.}

{\bf Subcase 2.1: $p\neq 0$.}
Then $(p-\be)a_{p}=0$. Since $a_{p}\neq 0$, we have $p=\be$.

If $q=p=\beta$, then we see that $-\beta v_{0}\otimes v_{1}=-a_{\beta}v_{\beta+1}\otimes v_{-\beta}$. Therefore, $\beta=-1$ and $a_{\beta}=\beta=-1$. It is clear that
$$L_n \cdot (-v_{-1}\otimes v_1) \neq \sum\limits_{i=0}^{n-1}(i-n\beta)v_i \otimes v_{-i+n}, \quad n\geq 1.$$
Thus, we conclude that $\d_{1-\be,\be}=\d_{2,-1}$ is a non-trivial 1-cocycle.

If $q>p=\beta$, then we express $-\beta v_0 \otimes v_1$ as follows:
$$-\beta v_{0}\otimes v_{1}=\sum\limits_{i=\beta+1}^{q}(i-\beta)(a_{i}-a_{i-1})v_{i}\otimes v_{-i+1}-(q+1-\beta)a_qv_{q+1}\otimes v_{-q}.$$
Since $(q+1-\be)a_q\neq 0$, it follows that $q=-1$ and $a_i=-1$ for $i=\be,\cdots, -1$. This implies that $\be\leq-1$. Furthermore,
$$L_{n}\cdot\left(-\sum\limits_{i=\be}^{-1}v_{i}\otimes v_{-i}\right)=-\sum\limits_{i=\be}^{-1}(i-n\be)v_{i}\otimes v_{-i+n}+\sum\limits_{i=\be}^{-1}(i+n-n\be)v_{i+n}\otimes v_{-i}.$$
However, while
$$L_{\be}\cdot\left(-\sum\limits_{i=\be}^{-1}v_{i}\otimes v_{-i}\right)\neq-\sum\limits_{i=\be}^{-1}(i-\be^{2})v_{i}\otimes v_{-i+\be},$$
it is impossible.

{\bf Subcase 2.2: $p=0$.}

If $p=0$ and $q>0$, we have $(q+1-\be)a_{q}=0$. Since $a_{q}\neq 0$, we have $q=\be-1$, which implies that $\be>1$ and $a_{i}=1$ for $i=0, \cdots, \be-1$. Then
$$\d_{1-\be, \be}(L_{n})=L_{n}\cdot\left(\sum\limits_{i=0}^{\be-1}v_{i}\otimes v_{-i}\right)=\sum\limits_{i=0}^{\be-1}(i-n\be)v_{i}\otimes v_{-i+n}-\sum\limits_{i=0}^{\be-1}(i+n-n\be)v_{i+n}\otimes v_{-i}.$$
While
$L_{\be}\cdot\left(\sum\limits_{i=0}^{\be-1}v_{i}\otimes v_{-i}\right)\neq
\sum\limits_{i=0}^{\be-1}(i-\be^{2})v_{i}\otimes v_{-i+\be},$
 it is impossible.

If $p=q=0$, we have
$$-\be v_{0}\otimes v_{1}=(-\be)a_{0}v_{0}\otimes v_{1}-(1-\be)a_{0}v_{1}\otimes v_{0}.$$
Then $\be=a_{0}=\be a_{0}.$ So $\be=a_{0}=0$ or  $\be=a_{0}=1$. Since $a_{0}\neq 0$, we have
$\be=a_{0}=1$. It is easy to see that $L_{n}\cdot(v_{0}\otimes v_{0})=-n v_{0}\otimes v_{n}\neq \d_{0, 1}(L_{n})$ for all $n\in\bZ$.
Hence $\d_{1-\be,\be}=\d_{0, 1}$ is a non-trivial 1-cocycle.

The proof is complete.
\end{proof}

\begin{prop}\label{P3.0}
The following linear operators $\d_{0,0}^i:\W\to \mathcal{F}_{0}\ot \mathcal{F}_{0} (i=1,\cdots 5)$ are non-trivial 1-cocycles of $\W$ with coefficients in $\mathcal{F}_{0}\ot \mathcal{F}_{0}$:
\begin{eqnarray*}
   && \d_{0,0}^{1}(L_{m})=v_{0}\otimes\,v_{m},\quad
   \d_{0,0}^{2}(L_{m})=mv_{0}\otimes\,v_{m},\\
   && \d_{0,0}^{3}(L_{m})=v_{m}\otimes\,v_{0},\quad \d_{0,0}^{4}(L_{m})=mv_{m}\otimes\,v_{0},\\
   & & \d_{0,0}^{5}(L_{n})=\begin{cases} v_{0}\otimes\,v_{0},& n=0;\\
   0,&n=1;\\
   -\sum\limits_{i=1}^{n-1}v_{i}\otimes\,v_{n-i}, &n\geq 2;\\
   \sum\limits_{i=n}^{0}v_{i}\otimes\,v_{n-i}, & n\leq -1
   \end{cases}
  \end{eqnarray*}
for $m,n\in\bZ$.
Furthermore, $\d_{0,0}^{1}, \d_{0,0}^{2},\d_{0,0}^{3},\d_{0,0}^{4},\d_{0,0}^{5}$ are linearly independent.
\end{prop}
\begin{proof}
It is easy to verify that  $\d_{0,0}^i$ (for $i=1,2,3,4$) defined above are non-trivial 1-cocycles of $\W$ in $\mathcal{F}_{0}\ot \mathcal{F}_{0}$. Next, we shall show that $\d_{0,0}^{5}$ is a non-trivial 1-cocycle. It is straightforward to verify $\d_{0,0}\in {\rm Der}(\W, \mathcal{F}_{0}\ot \mathcal{F}_{0})$.  Assume $\d_{0,0}^{5}$ is a 1-coboundary. It follows that
$$\d_{0,0}^{5}(L_{n})=L_{n}\cdot\left(\sum\limits_{i=p}^{q}a_{i}v_{i}\otimes v_{-i}\right)=\sum\limits_{i=p}^{q}ia_{i}\Big(v_{i}\otimes v_{-i+n}-v_{i+n}\otimes v_{-i}\Big)$$
for some $a_{i}\in\bC$, where $a_{p}, a_{q}\in\bC^{*}$ and $p\leq q$. In particular,
$$\d_{0,0}^{5}(L_{1})=pa_{p}v_{p}\otimes v_{-p+1}+\sum\limits_{i=p+1}^{q}\Big(ia_{i}-(i-1)a_{i-1}\Big)v_{i}\otimes v_{-i+1}-qa_{q}v_{q+1}\otimes v_{-q}.$$
Note that $\d_{0,0}^{5}(L_{1})=0$. Then $pa_{p}=qa_{q}=0, ia_{i}=(i-1)a_{i-1} $ for $ i=p+1,\cdots, q.$
That is, $ia_{i}=0$ for all $p\leq i\leq  q$. So $p=q=0$. While $\d_{0,0}^{5}(L_{n})=L_{n}\cdot(a_{0}v_{0}\otimes v_{0})=0$, we have
$\d_{0,0}^{5}(L_{n})=0$, which is a contradiction. Hence $\d_{0,0}^{5}$ is a non-trivial 1-cocycle.

Let $x_{1}\d_{0,0}^{1}+x_{2} \d_{0,0}^{2}+x_{3}\d_{0,0}^{3}+x_{4}\d_{0,0}^{4}+x_{5}\d_{0,0}^{5}=0$. By the both sides acting on $L_{1}, L_{2}$ respectively, we deduce $x_{i}=0, i=1,2,3,4,5$. So  $\d_{0,0}^{1}, \d_{0,0}^{2},\d_{0,0}^{3},\d_{0,0}^{4},\d_{0,0}^{5}$ are linearly independent.
\end{proof}

\begin{prop}\label{P3.1}

For $(\al,\be)=(0,1),(1,0),(0,2), (2,0), (1,1)$, the following linear operators $d_{\al,\be}:\W\to \mathcal{F}_{\al}\ot \mathcal{F}_{\be}$ are non-trivial 1-cocycles of $\W$:
\begin{eqnarray*}
                       \d_{0,1}^{1}(L_{n})&=&n^{2} v_{0}\otimes v_{n},\quad \d_{1,0}^{1}(L_{n})=n^{2} v_{n}\otimes\,v_{0},\\
\d_{0,2}(L_{n})&=&n^{3} v_{0}\otimes v_{n},\quad
 \d_{2,0}(L_{n})=n^{3} v_{n}\otimes v_{0},  \\
    \d_{1,1}(L_{n})&=&\begin{cases}
     \sum\limits_{i=1}^{n}i(n-i) v_{i}\otimes v_{n-i},&n\geq 1;\\
\quad\,     0,&n=0;\\
       -\sum\limits_{i=-1}^{n}i(n-i) v_{i}\otimes v_{n-i},& n\leq -1
   \end{cases}
\end{eqnarray*}
for $n\in \bZ$.  Moreover, $\d_{0,1}^{1}, \d_{0,1}$ are linearly independent, and $\d_{1,0}^{1}, \d_{1,0}$ are linearly independent.
\end{prop}
\begin{proof}
We directly verify that the linear operators ${\rm d}_{1,0},{\rm d}_{0,1},{\rm d}_{2,0},{\rm d}_{0,2}$ defined above are non-trivial 1-cocycles. Next, we shall show that  $\d_{1,1}$ is a non-trivial 1-cocycle. Assume $\d_{1,1}$ is a 1-coboundary, then
$$\d_{1,1}(L_{n})=L_{n}\cdot\left(\sum\limits_{i=p}^{q}a_{i}v_{i}\otimes v_{-i}\right)=\sum\limits_{i=p}^{q}a_{i}\Big((i-n)v_{i}\otimes v_{-i+n}-(i+n)v_{i+n}\otimes v_{-i}\Big)$$
for some $a_{i}\in\bC$, where $a_{p}, a_{q}\in\bC^{*}$ and $p\leq q$. In particular,
$$\d_{1,1}(L_{1})=(p-1)a_{p}v_{p}\otimes v_{-p+1}+\sum\limits_{i=p+1}^{q}\Big((i-1)a_{i}-ia_{i-1}\Big)v_{i}\otimes v_{-i+1}-(q+1)a_{q}v_{q+1}\otimes v_{-q}.$$
Notice that $\d_{1,1}(L_{1})=0$. It follows that $(p-1)a_{p}=(q+1)a_{q}=0$. So $p=1>q=-1$, which is a contradiction. Then $\d_{1,1}$ is a non-trivial 1-cocycle.
\end{proof}


\section{Classification of 1-cocycles for $(\al,\be)\notin\bZ^2$}\label{section 4}

In this section, we shall compute ${\rm H}^1(\W, \mathcal{F}_{\al}\ot \mathcal{F}_{\be})$  for $(\al,\be)\notin\bZ^2$.

Suppose $\be\notin\bZ$ without loss of generality.  By Lemma \ref{L1.2}-Lemma \ref{L1.3}, it suffices to consider $D\in\Der(\W, \mathcal{F}_{\al}\ot \mathcal{F}_{\be})_{0}$ and $D(L_0)=0$. Suppose
\begin{eqnarray*}
D(L_{n}) =\sum\limits_{i\in\Z} a_{n,i}v_i\!\otimes\!v_{n-i}, \quad\forall\,n\in\bZ^{*},
\end{eqnarray*}
where $a_{n, i}\in\C$.
According to  Lemma \ref{L1.4}, replace $D$
by $D-D_{v}$, where $D_{v}$ is a 1-coboundary induced by  $v=\sum v_{i}\otimes v_{-i}$.
We denote by
$$l=\max\{\,|i\,|\,\big|\,a_{1,i}\ne0\}.$$
For $\be\notin\bZ$,  we safely suppose that
\begin{equation}\label{1.12}
D(L_1)=a_{1,l+1}v_{l+1}\otimes\,v_{-l}.
\end{equation}

For the convenience of readers, we provide a  scheme of our calculation in this section:

\begin{itemize}
    \item If $D(L_{-1})=0$, we have $D=0$.
    \item If $D(L_{-1})\neq 0$. Set
$$p_{-1}=\min\{i\in\bZ\mid a_{-1,i}\neq 0\}, \quad q_{-1}=\max\{i\in\bZ\mid a_{-1,i}\neq 0\}.$$
By discussing the values of $p_{-1}$ and $q_{-1}$, we shall show that
$
 D\in  \Inn(\W,\mathcal{F}_{\al}\ot \mathcal{F}_{\be})_{0}\oplus \bC \de_{\al+\be,1}\d_{\al, \be}.
$

\end{itemize}

Let $m=1$ in (\ref{1.6ccc}). Then
\begin{eqnarray}
(l+1-\al)a_{1,l+1}&=&(l-1+\al)a_{-1,l-1}-(l+1-\be)a_{-1,l},    \label{1.13f1}\\
-(l+\be)a_{1,l+1}&=&(l+\al)a_{-1,l}-(l+2-\be)a_{-1,l+1},\label{1.13f2}\\
 (i-1+\al)a_{-1,i-1}&=&(i+1-\be)a_{-1,i}, \quad\quad\,i\neq l,l+1. \label{1.13f3}
\end{eqnarray}

\begin{lemm}\label{L1.10}
If $\be\notin\bZ$ and $D(L_{-1})=0$, we have $D=0$.
\end{lemm}
\begin{proof}
If $\be\notin\bZ$ and $D(L_{-1})=0$,
then we have $a_{1,l+1}=0$ from (\ref{1.13f2}), that is, $D(L_{1})=0$. Using (\ref{1.6c}), we have
$
 (i-2+\be)a_{2,i}=(i+1-\al)a_{2,i+1}.
$
Since $\{a_{2,i}\mid a_{2,i}\neq 0, i\in\bZ\}$ is finite and $\be\notin\bZ$,  we have $a_{2,i}=0$ for all $i\in\bZ$.
Then $D(L_{2})=0$. Using (\ref{1.6cc}), we have
$  (i-1+\al)a_{-2,i-1}=(i+2-\be)a_{-2,i}.
$
Since $\{a_{-2,i}\neq 0\}$ is finite and $\be\notin\bZ$,  we have $a_{-2,i}=0$ for all $i\in\bZ$.
Then $D(L_{-2})=0$. Since the Witt algebra $\W$ can be generated by
$\{L_{\pm 1}, L_{\pm 2}\}$, we have $D(L_{m})=0$ for all $m\in\bZ$. Then $D=0$.
\end{proof}

Next we  assume $D(L_{-1})\neq 0$. Set
$$p_{-1}=\min\{i\in\bZ\mid a_{-1,i}\neq 0\}, \quad q_{-1}=\max\{i\in\bZ\mid a_{-1,i}\neq 0\}.$$

\begin{lemm}\label{L1.5}
If $\be\notin\bZ$, we have $p_{-1}=l$ or $p_{-1}=l+1.$
Moreover, if $q_{-1}\neq l-1,l$, then $q_{-1}=-\al.$
\end{lemm}
\begin{proof}
Assume $p_{-1}\neq l,l+1$. Let $i=p_{-1}$ in (\ref{1.13f3}). From $a_{-1,p_{-1}-1}=0$, we have $a_{-1,p_{-1}}(\be-p_{-1}-1)=0.$
Since $\be\notin\bZ$, we have $a_{-1,p_{-1}}=0$, a contradiction. Therefore $p_{-1}=l $ or $ p_{-1}=l+1.$
Moreover, suppose $q_{-1}\neq l-1,l$, that is, $q_{-1}+1\neq l,l+1$. Let $i=q_{-1}+1$ in (\ref{1.13f3}). Then $ a_{-1,q_{-1}}(\al+q_{-1})=0$. It follows that $q_{-1}=-\al,$ which implies that  $\al\in\bZ$ and $-\al\neq l-1,l$.
\end{proof}

From Lemma \ref{L1.5}, we have the following result.
\begin{lemm}\label{L1.5+1}
If $\be\notin\bZ$ and $\al\notin\bZ$, we have $p_{-1}=q_{-1}=l.$
\end{lemm}

\begin{lemm}\label{L1.6}
If $\be\notin\bZ$ and $p_{-1}=q_{-1}=l$, we have
$
 D\in  \Inn(\W,\mathcal{F}_{\al}\ot \mathcal{F}_{\be})_{0}\oplus \bC \de_{\al+\be,1}\d_{\al, \be}.
$
\end{lemm}
\begin{proof}
If $\be\notin\bZ$ and $p_{-1}=q_{-1}=l$, that is, $D(L_{-1})=a_{-1,l}v_{l}\otimes v_{-l-1}  $, by (\ref{1.13f1})-(\ref{1.13f2}), we have
\begin{equation}\label{1.14}
a_{1,l+1}(\al-l-1)=-a_{-1,l}(\be-l-1),
\end{equation}
\begin{equation}\label{1.15}
a_{1,l+1}(\be+l)=-a_{-1,l}(\al+l).
\end{equation}
Since $a_{-1,l}\neq 0$ and $\be\notin\bZ$, we have $a_{1,l+1}\neq0$, $\al\neq l+1$ and
$(\al+l)(\al-l-1)=(\be+l)(\be-l-1).$
So $\al=\be$ or $\al+\be=1$.

{\bf Case 1:  $\al=\be\notin\bZ$.} By (\ref{1.15}), we have $a_{1,l+1}=-a_{-1,l}$,
$D(L_{1})=-a_{-1,l}v_{l+1}\otimes v_{-l}$ and $D(L_{-1})=a_{-1,l}v_{l}\otimes v_{-l-1}$.
Using (\ref{1.6c}), we have
\begin{eqnarray*}
-(l+1-2\be)a_{-1,l}&=&(l+1-\be)a_{2,l+1}-(l-2+\be)a_{2,l},     \\
-3a_{-1,l}&=&(l+2-\be)a_{2,l+2}-(l-1+\be)a_{2,l+1},            \\
 (l+2\be)a_{-1,l}&=&(l+3-\be)a_{2,l+3}-(l+\be)a_{2,l+2},  \\
 (i+1-\be)a_{2,i+1}&=&(i-2+\be)a_{2,i},\quad\quad\,i\neq l,l+1,l+2.
\end{eqnarray*}
Since $\{a_{2,i}\neq 0\}$ is finite and $\be\notin\bZ$,  we have $a_{2,i}=0$ for $i\leq l$ or $i\geq l+3$.
Then
\begin{eqnarray*}
&&-(l+1-2\be)a_{-1,l}=(l+1-\be)a_{2,l+1},\quad -3a_{-1,l}=(l+2-\be)a_{2,l+2}-(l-1+\be)a_{2,l+1},\\
&& (l+2\be)a_{-1,l}=-(l+\be)a_{2,l+2}.
\end{eqnarray*}
It follows that
$$a_{2,l+1}=-\frac{l+1-2\be}{l+1-\be}a_{-1,l},\quad\quad\,a_{2,l+2}=-\frac{l+2\be}{l+\be}a_{-1,l}, $$
$$\left(3+(l-1+\be)\frac{l+1-2\be}{l+1-\be}-(l+2-\be)\frac{l+2\be}{l+\be}\right)a_{-1,l}=0.$$
Since
$$\left(3+(l-1+\be)\frac{l+1-2\be}{l+1-\be}-(l+2-\be)\frac{l+2\be}{l+\be}\right)=-\frac{2\be(1-\be)(1-2\be)}{(l+\be)(l+1-\be)}$$
and $a_{-1,l}\neq 0$, we get $\be=\frac{1}{2}.$
Hence $a_{2,l+1}=-\frac{l}{l+\frac{1}{2}}a_{-1,l},\,a_{2,l+2}=-\frac{l+1}{l+\frac{1}{2}}a_{-1,l}$ and
\begin{eqnarray*}
&&D(L_{2})=-\frac{1}{l+\frac{1}{2}}a_{-1,l}\Big(lv_{l+1}\otimes v_{-l+1}+(l+1)v_{l+2}\otimes v_{-l}\Big).
\end{eqnarray*}
Since $\{a_{-2,i}\neq 0\}$ is finite, using (\ref{1.6cc}),  we have $a_{-2,i}=0$ for $i\neq l-1,l$, and
$$(l+\frac{1}{2})a_{-2,l-1}=la_{-1,l},\quad\,(l+\frac{3}{2})a_{-2,l}-(l-\frac{1}{2})a_{-2,l-1}=3a_{-1,l},\quad\,(l+\frac{1}{2})a_{-2,l}=(l+1)a_{-1,l}. $$
Hence
$a_{-2,l-1}=\frac{l}{l+\frac{1}{2}}a_{-1,l},\, a_{-2,l}=\frac{l+1}{l+\frac{1}{2}}a_{-1,l}.$
Set $a=-\frac{a_{-1,l}}{l+\frac{1}{2}}$. Then
$$D(L_{-1})=-a(l+\frac{1}{2})v_{l}\otimes v_{-l-1},\quad\,D(L_{-2})=-a\Big(l v_{l-1}\otimes v_{-l-1}+(l+1)v_{l}\otimes v_{-l-2}\Big),$$
$$D(L_{1})=a(l+\frac{1}{2})v_{l+1}\otimes v_{-l},\quad\,D(L_{2})=a\Big(l v_{l+1}\otimes v_{-l+1}+(l+1)v_{l+2}\otimes v_{-l}\Big).$$
Applying $D$ on the identities $[L_{k}, L_{1}]=(k-1)L_{k+1}$ and $[L_{-k}, L_{-1}]=(-k+1)L_{-k-1}$  and using induction on $k\in\mathbb{Z}_{+}$,  we  obtain
$$D(L_{n})=a\sum\limits_{i=l+1}^{l+n}(i-\frac{n}{2})v_{i}\otimes v_{-i+n}=L_{n}\cdot\left(-a{\sum\limits_{i=0}^{l}}v_{i}\otimes v_{-i}\right)+a\d_{\frac{1}{2},\frac{1}{2}}(L_{n}), \quad\, n\geq 1;$$
$$D(L_{n})=-a\sum\limits_{i=l+n+1}^{l}(i-\frac{n}{2})v_{i}\otimes v_{-i+n}=L_{n}\cdot\left(-a{\sum\limits_{i=0}^{l}}v_{i}\otimes v_{-i}\right)+a\d_{\frac{1}{2},\frac{1}{2}}(L_{n}), \quad\, n\leq -1.$$
Therefore, $D\in \Inn(\W,\mathcal{F}_{\al}\ot \mathcal{F}_{\be})_{0}\oplus \bC \d_{\frac{1}{2},\frac{1}{2}}$.

{\bf Case 2: $\al=1-\be$.} Then $a_{1,l+1}=-\frac{l+1-\be}{l+\be}a_{-1,l}$ from (\ref{1.15}) and
$D(L_{1})=-\frac{l+1-\be}{l+\be}a_{-1,l}v_{l+1}\otimes v_{-l}.$
Using (\ref{1.6c}), we have
\begin{eqnarray*}
-(l+1-2\be)a_{-1,l}&=&(l+\be)a_{2,l+1}-(l-2+\be)a_{2,l},\\
-3\frac{l+1-\be}{l+\be}a_{-1,l}&=&(l+1+\be)a_{2,l+2}-(l-1+\be)a_{2,l+1},\\
(l+2-2\be)a_{-1,l}&=&(l+2+\be)a_{2,l+3}-(l+\be)a_{2,l+2},\\
(i+\be)a_{2,i+1}&=&(i-2+\be)a_{2,i},\quad\quad\, i\neq l,l+1,l+2.
\end{eqnarray*}
Since $\{a_{2,i}\neq 0\}$ is finite,  $a_{2,i}=0$ for $i\leq l$ or $i\geq l+3$.
Then
\begin{eqnarray*}
&&a_{2,l+1}=-\frac{l+1-2\be}{l+\be}a_{-1,l},\quad a_{2,l+2}=-\frac{l+2-2\be}{l+\be}a_{-1,l},\\
&&D(L_{2})=-\frac{l+1-2\be}{l+\be}a_{-1,l}v_{l+1}\otimes v_{-l+1}-\frac{l+2-2\be}{l+\be}a_{-1,l}v_{l+2}\otimes v_{-l}.
\end{eqnarray*}
Using (\ref{1.6cc}), we have
$$a_{-2,l-1}=\frac{l-1+2\be}{l+\be}a_{-1,l},\quad a_{-2,l}=\frac{l+2\be}{l+\be}a_{-1,l}.$$
Set $a=-\frac{a_{-1,l}}{l+\be}$. Then
\begin{eqnarray*}
   D(L_{-1})&=&-a(l+\be)v_{l}\otimes v_{-l-1},\\
   D(L_{-2})&=&-a\Big((l-1+2\be)v_{l-1}\otimes v_{-l-1}+(l+2\be)v_{l}\otimes v_{-l-2}\Big),\\
   D(L_{1})&=&a(l+1-\be)v_{l+1}\otimes v_{-l},\\
   D(L_{2})&=&a\Big((l+1-2\be)v_{l+1}\otimes v_{-l+1}+(l+2-2\be)v_{l+2}\otimes v_{-l}\Big).
\end{eqnarray*}
Moreover, applying $D$ on the identities $[L_{k}, L_{1}]=(k-1)L_{k+1}$ and $[L_{-k}, L_{-1}]=(-k+1)L_{-k-1}$  and using induction on $k\in\mathbb{Z}_{+}$, we have
$$D(L_{n})=a\sum\limits_{i=l+1}^{l+n}(i-n\be)v_{i}\otimes v_{-i+n}=L_{n}\cdot\left(-a\sum\limits_{i=0}^{l}v_{i}\otimes v_{-i}\right)+a\sum\limits_{i=0}^{n-1}(i-n\be)v_{i}\otimes v_{-i+n}, \quad\, n\geq 1;$$
$$D(L_{n})=-a\sum\limits_{i=l+n+1}^{l}(i-n\be)v_{i}\otimes v_{-i+n}=L_{n}\cdot\left(-a\sum\limits_{i=0}^{l}v_{i}\otimes v_{-i}\right)-a\sum\limits_{i=n}^{-1}(i-n\be)v_{i}\otimes v_{-i+n}, \quad\, n\leq -1.$$
It follows that $D\in \Inn(\W,\mathcal{F}_{\al}\ot \mathcal{F}_{\be})_{0}\oplus \bC \d_{1-\be,\be}$.
\end{proof}

Next we discuss the case of $q_{-1}\geq l+1$. Notice that Lemma \ref{L1.5} implies  $\al=-q_{-1}\in\bZ$.

\begin{lemm}\label{L1.8+1}
If $\be\notin\bZ$, $p_{-1}=l$ and $q_{-1}=l+k_{0}$ for some integer $k_{0}\geq 1$,  then $D$ is a 1-coboundary.
\end{lemm}
\begin{proof}
If $\be\notin\bZ$, $p_{-1}=l$ and $q_{-1}=l+k_{0}\geq l+1$,  we have
$$D(L_{-1})=a_{-1,l}v_{l}\otimes v_{-l-1}+\cdots+a_{-1,l+k_{0}}v_{l+k_{0}}\otimes v_{-l-k_{0}-1},$$
$\al=-q_{-1}=-l-k_{0}$ by  { Lemma \ref{L1.4e1}} and
\begin{eqnarray*}
& &(2l+k_{0}+1)a_{1,l+1}=-(l+1-\be)a_{-1,l},\quad\,(l+\be)a_{1,l+1}=k_{0}a_{-1,l}+(l+2-\be)a_{-1,l+1},\\
& &(l+k_{0}+1-i)a_{-1,i-1}=-(i+1-\be)a_{-1,i}, \quad\quad\,i\neq l, l+1
\end{eqnarray*}
from (\ref{1.13f1})-(\ref{1.13f3}). Since $\be\notin\bZ$ and $a_{-1,l}\neq 0$, it is easy to see that
$2l+k_{0}+1\neq 0$ and
\begin{equation}\label{1.90}
a_{1,l+1}{=}-\frac{a_{-1,l}}{2l+k_{0}+1}(l+1-\be),
\end{equation}
\begin{equation}\label{1.91}
a_{-1,l+1}=-\frac{a_{-1,l}}{2l+k_{0}+1}\frac{(l+k_{0}+\be)(l+k_{0}+1-\be)}{l+2-\be},
\end{equation}
\begin{equation*}
a_{-1,i}=-\frac{l+k_{0}+1-i}{i+1-\be}a_{-1,i-1}, \quad\quad\,i\geq l+2.
\end{equation*}
Then
\begin{equation}\label{1.92}
a_{-1,l+i}=(-1)^{i-1}\left(\prod\limits_{t=1}^{i-1}\frac{k_{0}-t}{l+t+2-\be}\right)a_{-1,l+1}, \quad\quad\, 2\le i\leq k_{0}.
\end{equation}
Set $a=-\frac{a_{-1,l}}{2l+k_{0}+1}$. Then
$D(L_{1})=L_{1}\cdot\left(\sum\limits_{i=l+1}^{l+k_{0}}x_{i}v_{i}\otimes v_{-i}\right)$,
where
\begin{equation*}
x_{l+1}=a,\quad\,{x_{l+i}}=(-1)^{i-1}\left(\prod\limits_{t=2}^{i}\frac{k_{0}-t+1}{l+t-\be}\right)a, \quad\quad\, 2\le i\leq k_{0}.
\end{equation*}
In fact, $x_{l+1}, \cdots\, x_{l+k_{0}}$ satisfy the following relations:
\begin{equation*}
x_{i}=-\frac{l+k_{0}+1-i}{i-\be}x_{i-1},\quad\quad\, i=l+2,\cdots, l+k_{0}.
\end{equation*}
Moreover, we have
$$a_{1,l+1}=a(l+1-\be),\,a_{-1,l}=-a(2l+k_{0}+1),\,\, a_{-1,l+i}=\frac{(l+k_{0}+\be)(l+k_{0}+1-\be)}{l+i+1-\be}x_{l+i},\,\, 1\leq i\leq k_{0}.$$
It follows that $D(L_{-1})=L_{-1}\cdot \left(\sum\limits_{i=l+1}^{l+k_{0}}x_{i}v_{i}\otimes v_{-i}\right)$.
Denote by $D_{0}=D-D_{inn}$, where $D_{inn}(L_{n})=L_{n}\cdot\left(\sum\limits_{i=l+1}^{l+k_{0}}x_{i}v_{i}\otimes v_{-i}\right)$.
Then $D_{0}(L_{\pm 1})=0$.
By Lemma \ref{L1.10}, we have $D_{0}=0$, which implies that $D$ is a 1-coboundary.
\end{proof}

\begin{lemm}\label{L1.9}
If $\be\notin\bZ$, $p_{-1}=l+1$, and $q_{-1}=l+k_{0}$, where $k_{0}\geq 1$,  then $D$ is a 1-coboundary.
\end{lemm}
\begin{proof}
Suppose $\be\notin\bZ$,  $p_{-1}=l+1$,
and  $q_{-1}=l+k_{0}$, where $k_{0}\geq 1$.  Then $\al=-l-k_{0}$ and (\ref{1.13f1})-(\ref{1.13f3}) become
\begin{equation}\label{1.58}
(2l+k_{0}+1)a_{1,l+1}=0,
\end{equation}
\begin{equation}\label{1.59}
(l+\be)a_{1,l+1}=(l+2-\be)a_{-1,l+1},
\end{equation}
\begin{equation}\label{1.57}
a_{-1,i}=-\frac{l+k_{0}+1-i}{i+1-\be}a_{-1,i-1},\quad\quad\, i\geq l+2.
\end{equation}
Since $\be\notin\bZ$ and $a_{-1,l+1}\neq 0$, we have $a_{1,l+1}\neq 0$ via (\ref{1.59}). Therefore, by (\ref{1.58}), we have $2l+k_{0}+1=0$, that is,
$k_{0}=-2l-1$. Hence $l\leq -1$ and $\al=l+1$. From (\ref{1.59})-(\ref{1.57}), we have
$a_{1,l+1}=\frac{l+2-\be}{l+\be}a_{-1,l+1}$
and
$$D(L_{1})=\frac{l+2-\be}{l+\be}a_{-1,l+1}v_{l+1}\otimes v_{-l}, \quad\, D(L_{-1})=\sum\limits_{i=l+1}^{-l-1}a_{-1,i}v_{i}\otimes v_{-i-l},$$
where $a_{-1,i}=\frac{l+i}{i+1-\be}a_{-1,i-1}$ for $ i\geq l+2$. Furthermore,
\begin{equation}\label{1.61}
a_{-1,l+1+k}=\left(\prod\limits_{t=1}^{k}\frac{2l+1+t}{l+2+t-\be}\right)a_{-1,l+1}, \quad\quad\, 1\leq k\leq -2l-1.
\end{equation}
In addition, $D(L_{1})=L_{1}\cdot
\left(\sum\limits_{i=l+1}^{-l-1}x_{i}v_{i}\otimes v_{-i}\right)$, where
\begin{equation*}
x_{l+1}=\frac{l+2-\be}{(l+\be)(l+1-\be)}a_{-1,l+1},\quad\,
x_{i}=\frac{l+i}{i-\be}x_{i-1}, \quad\quad\, i=l+2,\cdots, -l-1.
\end{equation*}
Then
\begin{eqnarray*}
x_{l+k}
&=&\frac{2l+k}{l+k-\be}\cdot\frac{2l+k-1}{l+k-1-\be}\cdots\frac{2l+3}{l+3-\be}\cdot\frac{2l+2}{l+2-\be}x_{l+1}=\frac{l+k+1-\be}{(l+\be)(l+1-\be)}a_{-1,l+k}
\end{eqnarray*}
for $ 1\leq k\leq -2l-1$. Hence $$
\frac{(l+\be)(l+1-\be)}{i+1-\be}x_{i}=a_{-1,i}$$
for $l+1\leq
i\leq -l-1$. It follows that
$$(l-i)x_{i+1}+(i+\be)x_{i}=a_{-1,i},\quad l+1\leq i\leq -l-1.$$
 Moreover,
\begin{eqnarray*}
L_{-1}\cdot \left(\sum\limits_{i=l+1}^{-l-1}x_{i}v_{i}\otimes v_{-i}\right)
&=&-\sum\limits_{i=l+1}^{-l-1}x_{i}\Big((i-\al)v_{i-1}\otimes v_{-i}
+(-i-\be) v_{i}\otimes
v_{-i-1}  \Big)    \\
&=&\sum\limits_{i=l+1}^{-l-1}x_{i}\Big((l+1-i)v_{i-1}\otimes v_{-i}+(i+\be)v_{i}\otimes v_{-i-1}\Big)         \\
&=&\sum\limits_{i=l+1}^{-l-1}\Big((l-i)x_{i+1}+(i+\be)x_{i}\Big)
v_{i}\otimes v_{-i-1}\\
&=&\sum\limits_{i=l+1}^{-l-1}a_{-1,i} v_{i}\otimes
v_{-i-1}=D(L_{-1}) .
\end{eqnarray*}
 Denote by
$D_{0}=D-D_{inn}$, where
$$
D_{inn}(L_{n})=L_{n}\cdot\left(\sum\limits_{i=l+1}^{-l-1}x_{i}v_{i}\otimes
v_{-i}\right).$$
Then $D_{0}(L_{\pm 1})=0$. By Lemma \ref{L1.10},
we have $D_{0}=0$. Hence $D$ is a 1-coboundary.
\end{proof}

Next, we state our main result in this section.
\begin{prop}\label{MP1}
If $(\al,\be)\notin\bZ^2$, then
$
 {\rm H}^1(\W,\mathcal{F}_{\al}\ot \mathcal{F}_{\be})= \bC \delta_{\al+\be,1}\d_{\al,\be},
$
where $\d_{\al,\be}$ is defined  in Proposition \ref{P1.6--}.
\end{prop}

\begin{proof}
If $\be\notin\bZ$, by Lemma \ref{L1.5+1}-Lemma \ref{L1.9},  we have
\begin{align*}
 \Der(\W, \mathcal{F}_{\al}\ot \mathcal{F}_{\be})_{0}=
\begin{cases}
   \Inn(\W, \mathcal{F}_{\al}\ot \mathcal{F}_{\be})_{0}\oplus \bC \d_{\al,\be},&\al+\be=1;\\
   \Inn(\W, \mathcal{F}_{\al}\ot \mathcal{F}_{\be})_{0},& \text{otherwise}.
   \end{cases}
\end{align*}
It follows from Lemma \ref{L1.1} and Lemma \ref{L1.2} that $ {\rm H}^1(\W,\mathcal{F}_{\al}\ot \mathcal{F}_{\be})= \bC \delta_{\al+\be,1}\d_{\al,\be}$. The case $\alpha \notin \mathbb{Z}$ can be handled in a similar manner.

\end{proof}

\section{Classification of 1-cocycles for  $(\al,\be)\in\bZ^2$ }\label{section 5}

In this section, we shall compute ${\rm H}^1(\W, \mathcal{F}_{\al}\ot \mathcal{F}_{\be})$  for the case $(\al,\be)\in\bZ^2$.

By Lemma \ref{L1.2}-Lemma \ref{L1.3}, it suffices to consider
$D\in\Der(\W, \mathcal{F}_{\al}\ot \mathcal{F}_{\be})_{0}$ and $
D(L_{0})=a \delta_{\al,0}\delta_{\be,0} v_{0}\otimes v_{0}
$
for some $a\in\bC$. If $D(L_{0})\neq 0$,  we suppose that $D(L_{0})=v_{0}\otimes\,v_{0}$. By  replacing $D$ by $D-\d_{0,0}^{1}$, we have $D(L_{0})=0$.  Thus it suffices to deal with the case $D(L_{0})= 0$.

By  Lemma \ref{L1.4e1} and Lemma \ref{L1.4e4}, we assume
$$D(L_1)=a_{1,\be}v_{\be}\otimes\,v_{-\be+1}+a_{1,l+1}v_{l+1}\otimes\,v_{-l}$$
for some $l\in\bZ$ and $\be\leq l$.

For the convenience of readers, we provide a  scheme of our proof in this section. It will be divided into three cases for discussion.
\begin{itemize}
      \item If $D(L_{\pm1})=0$, we shall classify all 1-cocycles.
\item If $D(L_{\mp 1})=0$ and $D(L_{\pm1})\neq 0$, we shall determine all 1-cocycles.

\item If $D(L_{\pm1})\neq 0$ and
$$p_{-1}=\min\{i\mid a_{-1,i}\neq 0\}, \quad\, q_{-1}=\max\{i\mid a_{-1,i}\neq 0\},$$
By discussing the values of $p_{-1}$ and $q_{-1}$, we shall  obtain all 1-cocycles.

\end{itemize}

It is  important to note that the proof process for these three cases is not entirely analogous to others and needs individual treatment.

\subsection{$D(L_{\pm1})=0$}

\begin{lemm}\label{L3.0}
If $D(L_{0})=D(L_{\pm1})=0$, we have
\begin{align*}
 D\in
\begin{cases}
   \Inn(\W, \mathcal{F}_{\al}\ot \mathcal{F}_{\be})_{0}\oplus\bC\d_{0,0}^{1}\oplus\bC\d_{0,0}^{2}\oplus\bC\d_{0,0}^{3}\oplus\bC\d_{0,0}^{4}\oplus\bC\d_{0,0}^{5},&(\al,\be)=(0,0);\\
   \Inn(\W, \mathcal{F}_{\al}\ot \mathcal{F}_{\be})_{0}\oplus\bC\d_{0,1}\oplus\bC\d_{0,1}^{1},& (\al,\be)=(0,1);\\
   \Inn(\W, \mathcal{F}_{\al}\ot \mathcal{F}_{\be})_{0}\oplus\bC\d_{1,0}\oplus\bC\d_{1,0}^{1},&(\al,\be)=(1,0);\\
   \Inn(\W, \mathcal{F}_{\al}\ot \mathcal{F}_{\be})_{0}\oplus\bC\d_{0,2},& (\al,\be)=(0,2);\\
   \Inn(\W, \mathcal{F}_{\al}\ot \mathcal{F}_{\be})_{0}\oplus\bC\d_{2,0},&(\al,\be)=(2,0);\\
   \Inn(\W, \mathcal{F}_{\al}\ot \mathcal{F}_{\be})_{0}\oplus\bC\d_{1,1},& (\al,\be)=(1,1);\\
   \Inn(\W, \mathcal{F}_{\al}\ot \mathcal{F}_{\be})_{0},& \text{otherwise}.
\end{cases}
\end{align*}
\end{lemm}
\begin{proof}
If $D(L_{\pm1})=0$,  we have
\begin{equation}\label{96}
(i+1-\al)a_{2,i+1}=(i-2+\be)a_{2,i}, \quad\, (i-1+\al)a_{-2,i-1}=(i+2-\be)a_{-2,i}
\end{equation}
from (\ref{1.6c}) and (\ref{1.6cc}). Since both $\{a_{2,i}\neq 0|i\in\bZ\}$ and $\{a_{-2,i}\neq 0|i\in\bZ\}$ are finite,
by (\ref{96}), we have
$$D(L_{2})=\sum\limits_{i=\al}^{2-\be} a_{2,i}v_{i}\otimes v_{-i+2}, \quad\,   D(L_{-2})=\sum\limits_{i=\be-2}^{-\al} a_{-2,i}v_{i}\otimes v_{-i-2}, $$
 where $\al\leq 2-\be$. Moreover,
\begin{equation}\label{97}
a_{2,\al+k}= a_{2,\al}\prod\limits_{t=1}^{k}\frac{\al+\be-3+t}{t}, \quad\quad\, 1\leq k\leq 2-\al-\be,
\end{equation}
\begin{equation}\label{98}
 a_{-2,-\al-k}= a_{-2,-\al}\prod\limits_{t=1}^{k}\frac{\al+\be-3+t}{t}, \quad\quad\, 1\leq k\leq 2-\al-\be.
\end{equation}
On the other hand, by $L_{2}\cdot D(L_{-2})=L_{-2}\cdot D(L_{2})$, we have
\begin{eqnarray}
& &\be a_{-2,\be-2}v_{\be-2}\otimes v_{2-\be}+(\be-1) a_{-2,\be-1}v_{\be-1}\otimes v_{1-\be}   \nonumber\\
& &+ \sum\limits_{i=\be}^{-\al}\Big((i-2+2\al)a_{-2,i-2}+(-i-2+2\be) a_{-2,i}\Big)v_{i}\otimes v_{-i}  \nonumber\\
& &+(\al-1) a_{-2,-\al-1}v_{-\al+1}\otimes v_{\al-1}+\al a_{-2,-\al}v_{-\al+2}\otimes v_{\al-2}  \nonumber\\
&=&-\al a_{2,\al}v_{\al-2}\otimes v_{-\al+2}-(\al-1) a_{2,\al+1}v_{\al-1}\otimes v_{-\al+1}  \label{99}\\
& &+ \sum\limits_{i=\al}^{-\be}\Big((i+2-2\al)a_{2,i+2}+(-i+2-2\be) a_{2,i}\Big)v_{i}\otimes v_{-i}     \nonumber\\
& &-(\be-1) a_{2,-\be+1}v_{-\be+1}\otimes v_{\be-1}-\be a_{2,-\be+2}v_{-\be+2}\otimes v_{\be-2}.  \nonumber
\end{eqnarray}

{\bf Case 1: }  $\al< \be$. Then
$-\al a_{2,\al}v_{\al-2}\otimes v_{-\al+2}=0$ and $\al a_{-2,-\al}v_{-\al+2}\otimes v_{\al-2}=0$ via (\ref{99}). Hence
$\al a_{2,\al}=\al a_{-2,-\al}=0.$

{\bf Subcase 1.1:}  $\al\neq 0$. Then $a_{2,\al}=a_{-2,-\al}=0$, which forces $D(L_{\pm2})=0$. Hence $D=0$.

{\bf Subcase 1.2:} $\al=0$.  Since $\be\leq 2$, we have $0<\be\leq 2$, i.e., $\be=1$ or $2$.

{\bf (1) } $(\al,\be)=(0,1)$. By (\ref{96}), we get
$$D(L_{2})=a_{2,0}v_{0}\otimes v_{2}+a_{2,1}v_{1}\otimes v_{1}, \quad\,   D(L_{-2})=a_{-2,-1}v_{-1}\otimes v_{-1}+a_{-2,0}v_{0}\otimes v_{-2},$$
and
$a_{2,0}=-a_{2,1}, a_{-2,0}=-a_{-2,-1}.$
Let $m=2, i=-1$ in (\ref{1.6ccc}). Then $a_{-2,-1}=a_{2,1}$. Hence
$$D(L_{2})=a_{2,0}(v_{0}\otimes v_{2}-v_{1}\otimes v_{1}), \quad\,   D(L_{-2})=-a_{2,0}(v_{-1}\otimes v_{-1}-v_{0}\otimes v_{-2}).$$
Set $a=\frac{a_{2,0}}{2}$ and replace $D$ by $D{-}2a \d_{0,1}-D_{inn}$, where ${D_{inn}(L_{n})=L_{n}\cdot (av_{0}\otimes\,v_{0}})$ for all $n\in\bZ$. Then
$${D(L_{-2})=0,\quad\,   D(L_{-1})=-av_{0}\otimes\,v_{-1},\quad\, D(L_{1})=3av_{0}\otimes\,v_{1}, \quad\, D(L_{2})=8av_{0}\otimes\,v_{2}}.$$
Using induction on $n$ by
$(n-2)D(L_{n})=D[L_{n-1}, L_{1}]$ and $(n-2)D(L_{-n})=D[L_{-1}, L_{-n+1}]$,
we have
$$D(L_{n})=a({n^{2}+2n}) v_{0}\otimes v_{n}=a\d_{0,1}^{1}(L_{n}){-}L_{n}\cdot (2av_{0}\otimes v_{0}), \quad\forall n\in\bZ.$$
So $D\in \Inn(\W, \mathcal{F}_{\al}\ot \mathcal{F}_{\be})_{0}\oplus\bC\d_{0,1}\oplus\bC\d_{0,1}^{1}$.

{\bf (2)} $(\al,\be)=(0,2)$.  Then $D(L_{2})= a_{2,0}v_{0}\otimes v_{2} $ and $ D(L_{-2})= a_{-2,0}v_{0}\otimes v_{-2}.$
By (\ref{1.6ccc}), we have
$a_{-2,0}=-a_{2,0}$.
Using induction on $n$ by
$$(n-2)D(L_{n})=D[L_{n-1}, L_{1}]=-L_{1}\cdot D(L_{n-1}),\ (n-2)D(L_{-n})=D[L_{-1}, L_{-n+1}]=L_{-1}\cdot D(L_{-n+1}),$$
we obtain
$$D(L_{n})=\frac{a_{-2,0}}{6}(n^{3}-n) v_{0}\otimes v_{n}=\frac{a_{-2,0}}{6}\d_{0,2}(L_{n})+ L_{n}\cdot\left(\frac{a_{-2,0}}{12}v_{0}\otimes v_{0}\right), \quad\forall n\in\bZ.$$
So  $D\in \Inn(\W, \mathcal{F}_{\al}\ot \mathcal{F}_{\be})_{0}\oplus\bC\d_{0,2}$.


{\bf Case 2: } $\be<\al$.  Since $\al+\be\leq 2$, we have $2\be<\al+\be\leq 2$. So $\be\leq 0$. By (\ref{99}), we have
$$(\be-1) a_{-2,\be-1}v_{\be-1}\otimes v_{1-\be}=0, \quad\, -(\be-1) a_{2,-\be+1}v_{-\be+1}\otimes v_{\be-1}=0.$$
Furthermore,
$$\Big((i-2+2\al)a_{-2,i-2}+(-i-2+2\be) a_{-2,i}\Big)v_{i}\otimes v_{-i}=0, \quad\quad\, i=\be, \cdots, -\al,$$
$$\Big((i+2-2\al)a_{2,i+2}+(-i+2-2\be) a_{2,i}\Big)v_{i}\otimes v_{-i}=0, \quad\quad\, i=\al, \cdots, -\be,$$
$$(\al-1) a_{-2,-\al-1}v_{-\al+1}\otimes v_{\al-1}=0, \quad\, -(\al-1) a_{2,\al+1}v_{\al-1}\otimes v_{-\al+1}=0,$$
$$\al a_{-2,-\al}v_{-\al+2}\otimes v_{\al-2}=0, \quad\,  -\al a_{2,\al}v_{\al-2}\otimes v_{-\al+2}=0.$$

{\bf Subcase 2.1:} $\be<0$. Then $a_{-2,\be-2}=a_{-2,\be-1}=0$, $a_{2,-\be+2}=a_{2,-\be+1}=0$ and we get $D(L_{\pm2})=0$. Hence $D=0$.

{\bf Subcase 2.2:} $\be=0$. Then $\al=1$ or $2$.

{\bf (1) } $(\al,\be)=(1,0)$. Then
$$D(L_{2})=a_{2,1}v_{1}\otimes v_{1}+a_{2,2}v_{2}\otimes v_{0}, \quad\,
D(L_{-2})=a_{-2,-2}v_{-2}\otimes v_{0}+a_{-2,-1}v_{-1}\otimes v_{-1}.$$
By (\ref{96}), we have $a_{2,2}=-a_{2,1}, a_{-2,-2}=-a_{-2,-1}$.
Let $m=2, i=1$ in (\ref{1.6ccc}). Then
$a_{-2,-1}=a_{2,1}.$  So
$D(L_{2})=a_{2,2}(v_{2}\otimes v_{0}-v_{1}\otimes v_{1}) $ and $
D(L_{-2})=a_{2,2}(v_{-2}\otimes v_{0}-v_{-1}\otimes v_{-1}).$
Replacing $D$ by $D+a_{2,2} \d_{1,0}+D_{inn}$, where $D_{inn}(L_{n})=L_{n}\cdot\left(\frac{a_{2,2}}{2}v_{0}\otimes\,v_{0}\right)$ for all $n\in\bZ$. Then
$$D(L_{-1})=\frac{3a_{2,2}}{2}v_{-1}\otimes\,v_{0}, \quad\, D(L_{1})=-\frac{a_{2,2}}{2}v_{1}\otimes\,v_{0}, \quad\,D(L_{-2})=4a_{2,2}v_{-2}\otimes\,v_{0}, \quad\,  D(L_{2})=0.$$
Using induction on $n$ by $(n-2)D(L_{n})=D[L_{n-1}, L_{1}]$ and $(n-2)D(L_{-n})=D[L_{-1}, L_{-n+1}],$  we have
$$D(L_{n})=\frac{a_{2,2}}{2}(n^{2}-2n) {v_{n}\otimes v_{0}}=\frac{a_{2,2}}{2}\d_{0,1}^{1}(L_{n})+L_{n}\cdot \Big(a_{2,2}v_{0}\otimes v_{0}\Big), \quad\forall n\in\bZ.$$
Therefore, $D\in \Inn(\W, \mathcal{F}_{\al}\ot \mathcal{F}_{\be})_{0}\oplus\bC\d_{1,0}\oplus\bC\d_{1,0}^{1}$.

{\bf (2) } $(\al,\be)=(2,0)$. Then
$D(L_{2})= a_{2,2}v_{2}\otimes v_{0}$ and $D(L_{-2})= a_{-2,-2}v_{-2}\otimes v_{0}.$ Let $m=2,i=0$ in (\ref{1.6ccc}). Then
$a_{-2,-2}=-a_{2,2}$.
Using induction on $n$ by
$$(n-2)D(L_{n})=D[L_{n-1}, L_{1}]=-L_{1}\cdot D(L_{n-1}),(n-2)D(L_{-n})=D[L_{-1}, L_{-n+1}]=L_{-1}\cdot D(L_{-n+1}),$$
we deduce that
$$D(L_{n})=\frac{a_{2,2}}{6}(n^{3}-n) {v_{n}\otimes v_{0}}=\frac{a_{2,2}}{6}\d_{2,0}(L_{n})+ L_{n}\cdot\left(\frac{a_{2,2}}{12}v_{0}\otimes v_{0}\right), \quad\forall n\in\bZ,$$
which implies that $D\in \Inn(\W, \mathcal{F}_{\al}\ot \mathcal{F}_{\be})_{0}\oplus\bC\d_{2,0}$.

{\bf Case 3: }  $\al=\be$. Then $\al=\be=0, 1$ or $\al=\be< 0$.

{\bf Subcase 3.1:} $\al=\be=0$. Then
$D(L_{2})=\sum\limits_{i=0}^{2} a_{2,i}v_{i}\otimes v_{-i+2}$ and $ D(L_{-2})=\sum\limits_{i=-2}^{0} a_{-2,i}v_{i}\otimes v_{-i-2}$.
By (\ref{97})-(\ref{98}), we have
$a_{2,1}=-2a_{2,0}, a_{2,2}=a_{2,0}, a_{-2,-1}=-2a_{-2,0}, a_{-2,-2}=a_{-2,0}.$
Using (\ref{99}), we get $a_{-2,-1}=-a_{2,1}$. So $a_{-2,0}=-a_{2,0}$. Set $a=a_{2,0}$, then
$$D(L_{2})=a(v_{0}\otimes v_{2}+v_{2}\otimes v_{0})-2a v_{1}\otimes v_{1},\quad D(L_{-2})=-a(v_{0}\otimes v_{-2}+v_{-2}\otimes v_{0})+2a v_{-1}\otimes v_{-1}.$$
Using induction on $n$ by $(n-2)D(L_{n})=D[L_{n-1}, L_{1}]$ and $(n-2)D(L_{-n})=D[L_{-1}, L_{-n+1}],$ we get
$$D(L_n)=\begin{cases}(n-1)a(v_{0}\otimes\,v_{n}+v_{n}\otimes\,v_{0})-2a \sum\limits_{i=1}^{n-1}v_{i}\otimes\,v_{n-i},& n\geq 2;\\
D(L_n)=(n+1)a(v_{n}\otimes\,v_{0}+v_{0}\otimes\,v_{n})+2a \sum\limits_{i=n+1}^{-1}v_{i}\otimes\,v_{n-i},& n\leq -2.
\end{cases}$$
Then  $D=a \left(-\d_{0,0}^{1}+\d_{0,0}^{2}-\d_{0,0}^{3}+\d_{0,0}^{4}+2\d_{0,0}^{5}\right).$ It follows that
 $$D\in \Inn(\W, \mathcal{F}_{\al}\ot \mathcal{F}_{\be})_{0}\oplus\bC\d_{0,0}^{1}\oplus\bC\d_{0,0}^{2}\oplus\bC\d_{0,0}^{3}\oplus\bC\d_{0,0}^{4}\oplus\bC\d_{0,0}^{5}.$$

{\bf Subcase 3.2:} $\al=\be=1$. Then
$$D(L_{2})= a_{2,1}v_{1}\otimes v_{1}, \quad\,   D(L_{-2})= a_{-2,-1}v_{-1}\otimes v_{-1}=- a_{2,1}v_{-1}\otimes v_{-1}.$$
Using induction on $n$ by
$$(n-2)D(L_{n})=D[L_{n-1}, L_{1}]=-L_{1}\cdot D(L_{n-1}),(n-2)D(L_{-n})=D[L_{-1}, L_{-n+1}]=L_{-1}\cdot D(L_{-n+1}),$$
 we obtain
$$D(L_{n})=a_{2,1}\sum\limits_{i=1}^{n-1}i(n-i) v_{i}\otimes v_{n-i}=a_{2,1}{\sum\limits_{i=0}^{n-1}}i(n-i) v_{i}\otimes v_{n-i}, \quad\quad\, n\geq 0;$$
$$D(L_{n})=-a_{2,1}\sum\limits_{i=n}^{0}i(n-i) v_{i}\otimes v_{n-i}=-a_{2,1}{\sum\limits_{i=n}^{-1}}i(n-i) v_{i}\otimes v_{n-i}, \quad\quad\, n\leq -1.$$
Then  $D=a_{2,1}\d_{1,1}\in  \Inn(\W, \mathcal{F}_{\al}\ot \mathcal{F}_{\be})_{0}\oplus\bC\d_{1,1}$.


{\bf Subcase 3.3:} $\al=\be<0$. Then $a_{-2,\al+k-2}=-a_{2,\al+k}$ for $0\leq k\leq -2\al+2$.
It follows that
$L_{\pm 1}\cdot\left(\sum\limits_{i=\al}^{-\al}x_{i}v_{i}\otimes v_{-i}\right)=0$
and
$D(L_{\pm 2})=L_{\pm2}\cdot\left(\sum\limits_{i=\al}^{-\al}x_{i}v_{i}\otimes v_{-i}\right),$
where
$x_{\al}=-\frac{a_{2,\al}}{\al},$  $x_{\al+k}=\left(\prod\limits_{t=1}^{k}\frac{2\al+t-1}{t}\right)x_{\al}$ for $1\leq k\leq -2\al.$
Hence  $D\in  \Inn(\W, \mathcal{F}_{\al}\ot \mathcal{F}_{\be})_{0}$.
\end{proof}

\subsection{$D(L_{\mp1})=0$ and $D(L_{\pm1})\neq 0$}

\begin{lemm}\label{L3.01}
If $D(L_{0})=D(L_{-1})=0$ and $D(L_{1})\neq 0$, we have
\begin{align*}
 D\in
\begin{cases}
   \Inn(\W, \mathcal{F}_{\al}\ot \mathcal{F}_{\be})_{0}\oplus \bC \d_{0,0}^{1}\oplus \bC\d_{0,0}^{2}\oplus \bC\d_{0,0}^{3}\oplus \bC\d_{0,0}^{4}\oplus \bC\d_{0,0}^{5},&(\al,\be)=(0,0);\\
   \Inn(\W, \mathcal{F}_{\al}\ot \mathcal{F}_{\be})_{0}\oplus \bC \d_{1,0}^{1}\oplus \bC \d_{1,0},& (\al,\be)=(1,0);\\
   \Inn(\W, \mathcal{F}_{\al}\ot \mathcal{F}_{\be})_{0}\oplus \bC \d_{1-\be,\be},& \al=1-\be, \, \be\leq -1;\\
   \Inn(\W, \mathcal{F}_{\al}\ot \mathcal{F}_{\be})_{0},& \text{otherwise}.
\end{cases}
\end{align*}
\end{lemm}
\begin{proof}
By (\ref{1.6ccc}), we have
\begin{equation}\label{104}
 (i+1-\al)a_{1,i+1}=(i-1+\be)a_{1,i}.
\end{equation}
Let $i=\be-1, l, l+1$ in (\ref{104}) respectively. Then
\begin{equation}\label{105}
 (\be-\al)a_{1,\be}=0, \quad\, (l+1-\al)a_{1,l+1}=(l-1+\be)a_{1,l}, \quad\, (l+\be)a_{1,l+1}=0.
\end{equation}

{\bf Case 1:} $\al\neq \be$. Then $a_{1,\be}=0$ by (\ref{105}) and $D(L_1)=a_{1,l+1}v_{l+1}\otimes\,v_{-l}$.
Since $D(L_{1})\neq 0$, we have $a_{1,l+1}\neq 0$.
Therefore, we have $a_{1,l}=0$ and $l=-\be, \,\al=1-\be$ via (\ref{105}).
{ For $\be\leq l$,  $\be\leq 0$ } and
$$D(L_1)=a_{1,-\be+1}v_{-\be+1}\otimes\,v_{\be}=a_{1,\al}v_{\al}\otimes\,v_{-\al+1}.$$
By (\ref{1.6c}), we have
\begin{equation}\label{106}
a_{2,-\be+2}+a_{2,-\be+1}=3a_{1,-\be+1},
\end{equation}
\begin{equation*}
(i+\be)a_{2,i+1}=(i-2+\be)a_{2,i},\quad\quad\, i\neq -\be+1.
\end{equation*}
Since $\{a_{2,i}\neq 0|i\in\bZ\}$ is finite, we have $a_{2,i}=0, i\neq -\be+1, -\be+2.$ By (\ref{1.6cc}), we have
\begin{equation}\label{107}
(1-2\be)a_{-2,-\be}-(3-2\be)a_{-2,-\be+1}=-\be a_{1,-\be+1},
\end{equation}
\begin{equation}\label{108}
(-1-2\be)a_{-2,-\be-2}-(1-2\be)a_{-2,-\be-1}=(\be-1) a_{1,-\be+1},
\end{equation}
\begin{equation}\label{109}
(i-\be)a_{-2,i-1}=(i+2-\be)a_{-2,i},\quad\quad\, i\neq -\be-1, -\be+1.
\end{equation}
It is clear that $a_{-2,\be}=a_{-2,\be-3}=0 $ and $\be a_{-2,-\be-1}=(\be-1)a_{-2,-\be}.$
By (\ref{1.6ccc}), we have
\begin{equation}\label{110}
 (i+2\be)a_{2,i+2}-(i-2+2\be)a_{2,i}=(i-2\be)a_{-2,i-2}-(i+2-2\be)a_{-2,i}.
\end{equation}
Furthermore, we have
\begin{eqnarray}
&& (\be-1)a_{2,-\be+1}=-(3\be+1)a_{-2,-\be-3}+(3\be-1)a_{-2,-\be-1},\label{111}\\
&& \be a_{2,-\be+2}=-3\be a_{-2,-\be-2}+(3\be-2)a_{-2,-\be}, \label{112}\\
&& -(\be-1)a_{2,-\be+1}=-(3\be-1)a_{-2,-\be-1}+(3\be-3)a_{-2,-\be+1},\label{113}\\
&&-\be a_{2,-\be+2}=-(3\be-2) a_{-2,-\be}+(3\be-4)a_{-2,-\be+2},\label{114}\\
&&(i-2\be)a_{-2,i-2}=(i-2\be+2)a_{-2,i}, \quad\, i\neq -\be-1,-\be, -\be+1,-\be+2. \label{115}
\end{eqnarray}

{\bf Subcase 1.1: $\be=0$}. Then $\al=1$. From (\ref{107})-(\ref{109}), we get
$$a_{-2,-2}+a_{-2,-1}=a_{1,1}, \quad\, a_{-2,i}=0, \quad\quad\, i\neq -2,-1.$$
Using (\ref{111}), we get
$a_{-2,-1}=a_{2,1}$.
Set $a=a_{1,1}, b=a_{2,1}$. Then
\begin{eqnarray*}
&&  D(L_{-1})=0, \quad\, D(L_1)=a_{1,1}v_{1}\otimes\,v_{0}=av_{1}\otimes\,v_{0},\\
&& D(L_2)=bv_{1}\otimes\,v_{1}+(3a-b)v_{2}\otimes\,v_{0},\quad D(L_{-2})=(a-b)v_{-2}\otimes\,v_{0}+bv_{-1}\otimes\,v_{-1},
\end{eqnarray*}
where $a_{2,1}+a_{2,2}=3a_{1,1}.$
Using induction on $n$ by $(n-2)D(L_{n})=D[L_{n-1}, L_{1}]$  and $(n-2)D(L_{-n})=D[L_{-1}, L_{-n+1}],$  we obtain
$$D(L_n)=\begin{cases} b\sum\limits_{i=1}^{n}iv_{i}\otimes\,v_{n-i}+\frac{a-b}{2}(n^{2}+n)v_{n}\otimes\,v_{0},& n\geq 1,\\
-b{\sum\limits_{i=n+1}^{-1}}iv_{i}\otimes\,v_{-i+n}+\frac{a-b}{2}(n^{2}+n)v_{n}\otimes\,v_{0},& n\leq -2.
\end{cases}$$
Furthermore,
\begin{eqnarray*}
D(L_n)
&=&b\sum\limits_{i=1}^{n-1}iv_{i}\otimes\,v_{n-i}+bnv_{n}\otimes\,v_{0}+\frac{a-b}{2}(n^{2}+n)v_{n}\otimes\,v_{0}\\
&=&b\sum\limits_{i=1}^{n-1}iv_{i}\otimes\,v_{n-i}
+\frac{a+b}{2}n^{2}v_{n}\otimes\,v_{0}
-\frac{a+b}{2}n v_{n}\otimes\,v_{0}\\
&=&b\d_{1,0}(L_{n})+\frac{a-b}{2}\d_{1,0}^{1}(L_{n})-\frac{a+b}{2}L_{n}\cdot(v_{0}\otimes\,v_{0}), \quad\,n\geq 1;\\
D(L_{0})&=&0=b\d_{1,0}(L_{0})+\frac{a-b}{2}\d_{1,0}^{1}(L_{0})
-\frac{a+b}{2}L_{0}\cdot(v_{0}\otimes\,v_{0});\\
D(L_{-1})&=&0=bv_{-1}\otimes\,v_{0}+\frac{a-b}{2}v_{-1}\otimes\,v_{0}
-\frac{a+b}{2}v_{-1}\otimes\,v_{0}\\
&=&b\d_{1,0}(L_{-1})+\frac{a-b}{2}\d_{1,0}^{1}(L_{-1})
-\frac{a+b}{2}L_{-1}\cdot(v_{0}\otimes\,v_{-1});
\\
D(L_{n})&=&-b{\sum\limits_{i=n+1}^{-1}}iv_{i}\otimes\,v_{-i+n}
+\frac{a-b}{2}(n^{2}+n)v_{n}\otimes\,v_{0}\\
&=&b\Big(\d_{1,0}(L_{n})+nv_{n}\otimes\,v_{0}\Big)+\frac{a-b}{2}\d_{1,0}^{1}(L_{n})
+\frac{a-b}{2}nv_{n}\otimes\,v_{0}\\
&=&b\d_{1,0}(L_{n})+\frac{a-b}{2}\d_{1,0}^{1}(L_{n})
+\frac{a+b}{2}nv_{n}\otimes\,v_{0}\\
&=&b\d_{1,0}(L_{n})+\frac{a-b}{2}\d_{1,0}^{1}(L_{n})-\frac{a+b}{2}L_{n}\cdot(v_{0}\otimes\,v_{0}),
\quad\, n\leq -2.
\end{eqnarray*}
Then ${{D=b\d_{1,0}+\frac{a-b}{2}\d_{1,0}^{1}-\frac{a+b}{2}D_{inn}}}$.
 where $D_{inn}(L_{n})=L_{n}\cdot(v_{0}\otimes\,v_{0})$ for all $n\in\bZ$.
So
$$D\in \Inn(\W, \mathcal{F}_{\al}\ot \mathcal{F}_{\be})_{0}\oplus \bC \d_{1,0}\oplus \bC \d_{1,0}^{1}.$$


{\bf Subcase 1.2: $\be=-1$}. Then $\al=2$.  Using (\ref{107})-(\ref{115}) and setting $a=\frac{a_{1,2}}{3}$, we obtain
$$D(L_{-1})=0, \quad\, D(L_1)=a(3v_{2}\otimes\,v_{-1}),$$
$$D(L_{-2})=a(2v_{0}\otimes\,v_{-2}+v_{1}\otimes\,v_{-3}),\quad\,D(L_2)=a(4v_{2}\otimes\,v_{0}+5v_{3}\otimes\,v_{-1}).$$
Using induction on $n$ by
$(n-2)D(L_{n})=D[L_{n-1}, L_{1}]$  and $(n-2)D(L_{-n})=D[L_{-1}, L_{-n+1}],$
 we  obtain
$$D(L_n)=a\sum\limits_{i=2}^{n+1}(n+i)v_{i}\otimes\,v_{n-i},\quad\, n\geq 1;\quad\,
D(L_{n})=-a\sum\limits_{i=n+2}^{1}(n+i)v_{i}\otimes\,v_{n-i},\quad\, n\leq- 2.$$
Then  $D=a\d_{2,-1}-a D_{inn}$,
 where $D_{inn}(L_{n})=L_{n}\cdot\left(\sum\limits_{i=0}^{1}v_{i}\otimes\,v_{-i}\right)$ for all $n\in\bZ$.
So
$$D\in \Inn(\W, \mathcal{F}_{\al}\ot \mathcal{F}_{\be})_{0}\oplus \bC \d_{2,-1}.$$


{\bf Subcase 1.3: $\be\leq -2$}. Then $\al=1-\be\geq3$, $\be<-\be-2$ and
$$\cdots\cdots, \be-3, \be-2, \be-1, \be, \cdots, -\be-2, -\be-1, -\be, -\be+1, \cdots\cdots.$$
Since $\{a_{-2,i}\neq 0|i\in\bZ\}$ is finite, by (\ref{109}), we deduce that
$$a_{-2,\be}=0,\quad\, a_{-2,i}=0, \quad\quad\, i\leq \be-3, \, or \, i\geq -\be+1.$$
Furthermore, we have $a_{-2,i}=0$ for $\be+1\leq i\leq -\be-2$. Use (\ref{107})-(\ref{109}), then we obtain
\begin{equation}\label{116}
a_{-2,\be-2}=-a_{-2,\be-1}, \quad\,a_{-2,-\be}=\frac{-\be}{1-2\be} a_{1,-\be+1},\quad\, a_{-2,-\be-1}=\frac{1-\be}{1-2\be} a_{1,-\be+1}.
\end{equation}
Let $i=\be, \be+1$ in (\ref{115}) respectively. Then  $a_{-2,\be-2}=a_{-2,\be-1}=0.$
Since $\be\leq -2$, we have $\be<-\be-3<-\be-2$. So $a_{-2,-\be-3}=0$. By (\ref{111})-(\ref{112}), we have
\begin{equation*}
 (\be-1)a_{2,-\be+1}=(3\be-1)a_{-2,-\be-1},\quad\, \be a_{2,-\be+2}=(3\be-2)a_{-2,-\be}.
\end{equation*}
Using (\ref{116}),
$$a_{2,-\be+1}=\frac{1-3\be}{1-2\be} a_{1,-\be+1},\quad\, a_{2,-\be+2}=\frac{2-3\be}{1-2\be} a_{1,-\be+1}.$$
Set $a=\frac{a_{1,-\be+1}}{1-2\be}$. It follows that $D(L_{-1})=0$ and
\begin{eqnarray*}
&&D(L_{-2})=a\Big(-\be v_{-\be}\otimes\,v_{\be-2}+(1-\be)v_{-\be-1}\otimes\,v_{\be-1}\Big),\quad D(L_1)=a\Big((1-2\be)v_{-\be+1}\otimes\,v_{\be}\Big),\\
&&D(L_2)=a\Big((1-3\be)v_{-\be+1}\otimes\,v_{\be+1}+(2-3\be)v_{-\be+2}\otimes\,v_{\be}\Big).
\end{eqnarray*}Using induction on $n$ by
$(n-2)D(L_{n})=D[L_{n-1}, L_{1}]$ and $(n-2)D(L_{-n})=D[L_{-1}, L_{-n+1}],$
we have
$$D(L_n)=a\sum\limits_{i=-\be+1}^{-\be+n}(i-n\be)v_{i}\otimes\,v_{n-i},\quad\, n\geq 1; \quad\, D(L_{n})=-a\sum\limits_{i=-\be+1+n}^{-\be}(i-n\be)v_{i}\otimes\,v_{n-i},\quad\, n\leq -1.$$
Then $D(L_{n})=a\d_{1-\be,\be}(L_{n})+L_{n}\cdot\left(-a\sum\limits_{i=0}^{-\be} v_{i}\otimes v_{-i}\right).$ So $D\in  \Inn(\W, \mathcal{F}_{\al}\ot \mathcal{F}_{\be})_{0}\oplus \bC \d_{1-\be,\be}$.


{\bf Case 2:} $\al=\be$. By (\ref{105}), we have
\begin{equation}\label{117}
(l+1-\be)a_{1,l+1}=(l-1+\be)a_{1,l}, \quad\, (l+\be)a_{1,l+1}=0.
\end{equation}

{\bf Subcase 2.1:} $l=-\be$. Then
$(1-2\be)a_{1,-\be+1}=-a_{1,-\be}$ via (\ref{117}) and
$$D(L_1)=a_{1,\be}v_{\be}\otimes\,v_{-\be+1}+a_{1,-\be+1}v_{-\be+1}\otimes\,v_{\be}.$$

Assume $a_{1,-\be+1}=0$. Then $a_{1,-\be}=0$ and $D(L_1)=a_{1,\be}v_{\be}\otimes\,v_{-\be+1}$, which suggests that $\be<-\be$, i.e., $\be<0$.
Using (\ref{1.6cc}), we have
\begin{equation}\label{126}
-\be a_{1,\be}=(2\be-3)a_{-2,\be-3},
\end{equation}
\begin{equation}\label{127}
(i-1+\be)a_{-2,i-1}=(i+2-\be)a_{-2,i}, \quad\quad\, i\neq \be-2, \be.
\end{equation}
Since $\{a_{-2,i}\neq 0|i\in\bZ\}$ is finite, by (\ref{127}), we have
$$a_{-2,i}=0, \quad\quad\, i\leq \be-3, \quad\, or \quad\, i\geq -\be+1.$$
Then  $a_{1,\be}=0$ via (\ref{126}). Hence $D(L_1)=0$, a contradiction.
It follows that $a_{1,-\be+1}\neq0.$ Hence $a_{1,-\be}\neq 0$, which implies that $-\be=\be$. Then
$\al=\be=l=0.$  Hence $D(L_1)=a_{1,0}v_{0}\otimes\,v_{1}+a_{1,1}v_{1}\otimes\,v_{0}$ and
$a_{1,0}=-a_{1,1}.$
By (\ref{1.6ccc})-(\ref{1.6cc}), we have
\begin{equation}\label{118}
3a_{1,i}=(i+1)a_{2,i+1}-(i-2)a_{2,i},
\end{equation}
\begin{equation}\label{119}
 (i+2)a_{1,i+2}-(i-1)a_{1,i}=(i-1)a_{-2,i-1}-(i+2)a_{-2,i},
\end{equation}
\begin{equation}\label{120}
 (i+2)(a_{2,i+2}+a_{-2,i})=(i-2)(a_{-2,i-2}+a_{2,i}).
\end{equation}
Using (\ref{118}), we have
\begin{equation}\label{121}
3a_{1,0}=a_{2,1}+2a_{2,0},\quad\, 3a_{1,1}=2a_{2,2}+a_{2,1},
\end{equation}
\begin{equation*}
(i+1)a_{2,i+1}=(i-2)a_{2,i},\quad\quad\,i\neq 0,1.
\end{equation*}
It is clear that $a_{2,i}=0$ for $ i\neq 0,1,2.$
Using (\ref{119}), we have
\begin{equation}\label{122}
-a_{1,1}=a_{-2,-1}+2a_{-2,-2},\quad\, -a_{1,0}=a_{-2,-1}+2a_{-2,0},
\end{equation}
\begin{equation*}
(i-1)a_{-2,i-1}=(i+2)a_{-2,i},\quad\quad\,i\neq -1,0.
\end{equation*}
It is clear that $a_{-2,i}=0$ for $ i\neq -2,-1,0.$
Let $i=-1,0$ in (\ref{120}) respectively. Then
\begin{equation}\label{123}
a_{-2,-1}=-a_{2,1},\quad a_{2,2}+a_{2,0}=-a_{-2,-2}-a_{-2,0}.
\end{equation}
Set $a=\frac{a_{1,0}}{2}, b=\frac{a_{2,1}}{2}$. Then  $a_{1,0}=2a, a_{1,1}=-2a$ and
$$a_{2,0}=3a-b,\quad\, a_{2,1}=2b,\quad\,a_{2,2}=-3a-b,\quad\, a_{-2,-2}=a+b,\quad\,a_{-2,-1}=-2b,\quad\, a_{-2,0}=b-a$$
via (\ref{121})-(\ref{123}). Therefore,
$$D(L_{-1})=0,\quad\, D(L_{-2})=(a+b)v_{-2}\otimes\,v_{0}-2b v_{-1}\otimes\,v_{-1} -(a-b)v_{0}\otimes\,v_{-2},$$
$$D(L_1)=2a v_{0}\otimes\,v_{1}-2a v_{1}\otimes\,v_{0},\quad\,D(L_{2})=(3a-b)v_{0}\otimes\,v_{2}+2b v_{1}\otimes\,v_{1}-(3a+b)v_{2}\otimes\,v_{0}.$$
Using induction on $n$ by
$(n-2)D(L_{n})=D[L_{n-1}, L_{1}] $ and $(n-2)D(L_{-n})=D[L_{-1}, L_{-n+1}],$
we have
$$D(L_n)=(n+1)a(v_{0}\otimes\,v_{n}-v_{n}\otimes\,v_{0})-(n-1)b(v_{0}\otimes\,v_{n}+v_{n}\otimes\,v_{0})+2b       \sum\limits_{i=1}^{n-1}v_{i}\otimes\,v_{n-i},\quad\, n\geq 2;$$
$$D(L_n)=-(n+1)a(v_{n}\otimes\,v_{0}-v_{0}\otimes\,v_{n})-(n+1)b(v_{n}\otimes\,v_{0}+v_{0}\otimes\,v_{n})-2b       \sum\limits_{i=n+1}^{-1}v_{i}\otimes\,v_{n-i},\quad\, n\leq -2.$$
Then
$$D=a\left(\d_{0,0}^{1}+\d_{0,0}^{2}-\d_{0,0}^{3}-\d_{0,0}^{4}\right)-b\left(-\d_{0,0}^{1}+\d_{0,0}^{2}-\d_{0,0}^{3}+\d_{0,0}^{4}+2\d_{0,0}^{5}\right).$$
Hence $D\in \Inn(\W, \mathcal{F}_{\al}\ot \mathcal{F}_{\be})_{0}\oplus \bC \d_{0,0}^{1}\oplus \bC\d_{0,0}^{2}\oplus \bC\d_{0,0}^{3}\oplus \bC\d_{0,0}^{4}\oplus \bC\d_{0,0}^{5}.$

{\bf Subcase 2.2:} $l\neq -\be$. By (\ref{117}), we have
$a_{1,l+1}=0$ and $(l-1+\be)a_{1,l}=0$. Then $a_{1,\be}\neq 0$,
$$D(L_1)=a_{1,\be}v_{\be}\otimes\,v_{-\be+1}.$$

If $\be<0$, we deduce $a_{1,\be}=0$ similar to the proof of {{ Subcase 2.1}}. So $\be\geq 0$.

If $\be=0$, similar to the proof of {{ Subcase 2.1}}, we get
$$ D(L_{2})=a_{2,0} v_{0}\otimes\,v_{2}+a_{2,1} v_{1}\otimes\,v_{1}+a_{2,2}v_{2}\otimes\,v_{0},$$
$$ D(L_{-2})=a_{-2,-2}v_{-2}\otimes\,v_{0}+a_{-2,-1} v_{-1}\otimes\,v_{-1}+a_{-2,0} v_{0}\otimes\,v_{-2},$$
where
$$a_{-2,-1}=-2a_{-2,-2}, \quad\,a_{-2,0}=a_{-2,-2}-\frac{a_{1,0}}{2},\quad\, a_{2,0}=a_{2,2}+\frac{3a_{1,0}}{2},\quad\,a_{2,1}=-2a_{2,2}.$$
Let $m=2, i=1$ in (\ref{1.6ccc}). Then $a_{-2,-1}=-a_{2,1}$. So $a_{-2,-2}=-a_{2,2}$. Let $m=2, i=0$ in (\ref{1.6ccc}). Then
$a_{2,2}+a_{-2,-2}=-a_{2,0}-a_{-2,0}$. Hence $a_{-2,0}+a_{2,0}=0$, which forces $a_{1,\be}=0$, a contradiction.

If $\be>0$, by (\ref{1.6c}), we obtain
\begin{equation}\label{128}
a_{2,\be+1}+(2-2\be)a_{2,\be}=3a_{1,\be},
\end{equation}
\begin{equation*}
(i+1-\be)a_{2,i+1}=(i-2+\be)a_{2,i},\quad\quad\, i\neq \be.
\end{equation*}
Since $\{a_{2,i}\neq 0|i\in\bZ\}$ is finite, we deduce
$a_{2,i}=0$ for ${i \leq \be-1}$ or ${i \geq 3-\be}$. {When $\be-1<3-\be$, we have $\be <2$. So $\be=1$ and  $a_{1,\be}=0$, which is a contradiction}.


\end{proof}

\begin{lemm}\label{L3.02}
If $D(L_{0})=D(L_{1})=0$ and $D(L_{-1})\neq 0$, we have
\begin{align*}
 D\in
\begin{cases}
   \Inn(\W, \mathcal{F}_{\al}\ot \mathcal{F}_{\be})_{0}\oplus \bC \d_{0,0}^{1}\oplus \bC\d_{0,0}^{2}\oplus \bC\d_{0,0}^{3}\oplus \bC\d_{0,0}^{4}\oplus \bC\d_{0,0}^{5},&(\al,\be)=(0,0);\\
    \Inn(\W, \mathcal{F}_{\al}\ot \mathcal{F}_{\be})_{0}\oplus \bC \d_{0,1}^{1}\oplus \bC \d_{0,1},&(\al,\be)=(0,1);\\
   \Inn(\W, \mathcal{F}_{\al}\ot \mathcal{F}_{\be})_{0}\oplus \bC \d_{1,0}^{1}\oplus \bC \d_{1,0},& (\al,\be)=(1,0);\\
   \Inn(\W, \mathcal{F}_{\al}\ot \mathcal{F}_{\be})_{0}\oplus \bC \d_{1-\be,\be},&{\be\in\bZ, \, \be\neq 0,1;}\\
   \Inn(\W, \mathcal{F}_{\al}\ot \mathcal{F}_{\be})_{0},& \text{otherwise}.
 \end{cases}
\end{align*}
\end{lemm}
\begin{proof} Similar to the proof of Lemma \ref{L3.01}.
\end{proof}

\subsection{ $D(L_{\pm1})\neq 0$}

Next, we always assume $D(L_{-1})\neq 0$ and denote
$$p_{-1}=\min\{i|a_{-1,i}\neq 0\}, \quad\, q_{-1}=\max\{i|a_{-1,i}\neq 0\}.$$

\begin{lemm}\label{L3.03}
If $D(L_{0})=0$ and $D(L_{\pm1})\neq 0$, we have
$$
 D\in
\begin{cases}
   \Inn(\W, \mathcal{F}_{\al}\ot \mathcal{F}_{\be})_{0}\oplus\bC\d_{0,0}^{1}\oplus\bC\d_{0,0}^{2}\oplus\bC\d_{0,0}^{3}\oplus\bC\d_{0,0}^{4}\oplus\bC\d_{0,0}^{5},& (\al,\be)=(0,0);\\
   \Inn(\W, \mathcal{F}_{\al}\ot \mathcal{F}_{\be})_{0}\oplus\bC\d_{0,1}\oplus\bC\d_{0,1}^{1},&(\al,\be)=(0,1);\\
   \Inn(\W, \mathcal{F}_{\al}\ot \mathcal{F}_{\be})_{0}\oplus\bC\d_{1,0}\oplus\bC\d_{1,0}^{1},& (\al,\be)=(1,0);\\
   \Inn(\W, \mathcal{F}_{\al}\ot \mathcal{F}_{\be})_{0}\oplus\bC\d_{0,2},& (\al,\be)=(0,2);\\
   \Inn(\W, \mathcal{F}_{\al}\ot \mathcal{F}_{\be})_{0}\oplus\bC\d_{2,0},& (\al,\be)=(2,0);\\
   \Inn(\W, \mathcal{F}_{\al}\ot \mathcal{F}_{\be})_{0}\oplus\bC\d_{1,1},&(\al,\be)=(1,1);\\
   \Inn(\W, \mathcal{F}_{\al}\ot \mathcal{F}_{\be})_{0},& \text{otherwise}.
   \end{cases}
$$
\end{lemm}
\begin{proof}
Let $m=1$ in (\ref{1.6ccc}). Then
\begin{equation}\label{131}
 (i+1-\al)a_{1,i+1}-(i-1+\be)a_{1,i}=(i-1+\al)a_{-1,i-1}-(i+1-\be)a_{-1,i}.
\end{equation}
Let $i=\be-1, \be, \be+1, l, l+1$. Then
\begin{equation}\label{132}
(\be-\al)a_{1,\be}=(\al+\be-2)a_{-1,\be-2},
\end{equation}
\begin{equation}\label{133}
(\be+1-\al)a_{1,\be+1}-(2\be-1)a_{1,\be}=(\al+\be-1)a_{-1,\be-1}-a_{-1,\be},
\end{equation}
\begin{equation}\label{134}
(\be+2-\al)a_{1,\be+2}-2\be a_{1,\be+1}=(\al+\be)a_{-1,\be}-2a_{-1,\be+1},
\end{equation}
\begin{equation}\label{135}
(l+1-\al)a_{1,l+1}-(l-1+\be) a_{1,l}=(l-1+\al)a_{-1,l-1}-(l+1-\be)a_{-1,l},
\end{equation}
\begin{equation}\label{136}
-(l+\be) a_{1,l+1}=(l+\al)a_{-1,l}-(l+2-\be)a_{-1,l+1},
\end{equation}
and
\begin{equation}\label{137}
(i-1+\al)a_{-1,i-1}=(i+1-\be)a_{-1,i}, \quad\quad\, i\leq \be-2, \quad\, or\quad\, i\geq l+2.
\end{equation}
Assume $p_{-1}\leq \be-2$. Let $i=p_{-1}$ in (\ref{137}). Then $p_{-1}=\be-1\geq \be-2$, which is a contradiction. Therefore, $p_{-1}\geq \be-1$, which suggests
$$a_{-1, i}=0, \quad\quad\, i\leq \be-2.$$
Hence (\ref{132}) becomes
\begin{equation}\label{138}
(\be-\al)a_{1,\be}=0.
\end{equation}

{\bf Case 1:} $\al\neq \be$. Then $a_{1,\be}=0$ and $D(L_1)=a_{1,l+1}v_{l+1}\otimes\,v_{-l}$, where $\be\leq l$.
Moreover,  by  (\ref{1.13f1})-(\ref{1.13f3}),
\begin{eqnarray*}
& &(l+1-\al)a_{1,l+1}=(l-1+\al)a_{-1,l-1}-(l+1-\be)a_{-1,l}, \\   
& &-(l+\be)a_{1,l+1}=(l+\al)a_{-1,l}-(l+2-\be)a_{-1,l+1},\\
& & (i-1+\al)a_{-1,i-1}=(i+1-\be)a_{-1,i}, \quad\quad\,i\neq l,l+1.
\end{eqnarray*}
Similar to the proof of Lemma \ref{L1.5}, we have the following result:
$$p_{-1}=\be-1, l, \, or\, l+1;\quad\,  q_{-1}=-\al   \quad\, if\,\, q_{-1}\neq l-1,l.$$

{\bf Subcase 1.1:} $p_{-1}=\be-1$ and $q_{-1}=l-1$. Then $a_{-1,i}=0$ for $i\geq l$. According to (\ref{1.13f1})-(\ref{1.13f2}), we have
\begin{equation}\label{139}
(l+1-\al)a_{1,l+1}=(l-1+\al)a_{-1,l-1},\quad\, (l+\be)a_{1,l+1}=0.
\end{equation}
Since $D(L_{1})\neq 0$, that is, $a_{1,l+1}\neq 0$, we have $\be=-l.$
Since $\be\leq l$, we have $\be\leq 0$ and
$$D(L_1)=a_{1,-\be+1}v_{-\be+1}\otimes\,v_{\be}, \quad\, D(L_{-1})=\sum\limits_{i=\be-1}^{-\be-1}a_{-1,i}v_{i}\otimes\,v_{-i-1}.$$

{\bf (1.1.1)} If $\be=0$ and $\al\neq 0, 1$, we have $a_{-1,-1}=-a_{1,1}$ and
$$D(L_1)=a_{1,1}v_{1}\otimes\,v_{0}=L_{1}\cdot\left(-\frac{a_{1,1}}{\al}v_{0}\otimes\,v_{0}\right), \quad\,
D(L_{-1})=-a_{1,1}v_{-1}\otimes\,v_{0}=L_{-1}\cdot\left(-\frac{a_{1,1}}{\al}v_{0}\otimes\,v_{0}\right).$$
Replace $D$ by $D-D_{inn}$, where $D_{inn}(L_{n})=L_{n}\cdot\left(-\frac{a_{1,1}}{\al}v_{0}\otimes\,v_{0}\right)$ for all $n\in\bZ$, then
$D(L_{0})=D(L_{\pm 1})=0$.  According to the result of  {Lemma } \ref{L3.0}, we know $D\in \Inn(\W, \mathcal{F}_{\al}\ot \mathcal{F}_{\be})_{0}\oplus\delta_{\al,2}\bC \d_{2,0}.$

{\bf (1.1.2)} If $\be=0$  and  $\al=1$ , we have
$$D(L_1)=L_{1}\cdot (-a_{1,1}v_{0}\otimes\,v_{0}), \, D(L_{-1})=L_{-1}\cdot (-a_{1,1}v_{0}\otimes\,v_{0})+(a_{1,1}+a_{-1,-1})v_{-1}\otimes\,v_{0}.$$
Set $a=a_{1,1}+a_{-1,-1}$. Then it suffices to consider $D(L_{0})=D(L_{1})=0 $ and $ D(L_{-1})=av_{-1}\otimes\,v_{0}.$
By  {Lemma } \ref{L3.0} and {Lemma } \ref{L3.02}, we have
$D\in \Inn(\W, \mathcal{F}_{\al}\ot \mathcal{F}_{\be})_{0}\oplus\bC \d_{1,0}\oplus\bC \d_{1,0}^{1}.$

{\bf (1.1.3)} If $\be<0$, we have
$$D(L_1)=a_{1,-\be+1}v_{-\be+1}\otimes\,v_{\be}, \quad\, D(L_{-1})=\sum\limits_{i=\be-1}^{-\be-1}a_{-1,i}v_{i}\otimes\,v_{-i-1},$$
where
\begin{equation}\label{142}
a_{-1,\be+i}=\left(\prod\limits_{t=1}^{i+1}\frac{\al+\be-2+t}{t}\right)a_{-1,\be-1},\quad\, 0\leq i\leq -2\be-1.
\end{equation}
Since $a_{-1, -\be-1}\neq 0$, we have $\al+\be+i\neq 0$ for $-1\leq i\leq -2\be-2$. Hence $\al-i\neq 0$ for $\be+2\leq i\leq -\be+1$.
Moreover, by (\ref{139}), we have
\begin{equation}\label{143}
a_{1,-\be+1}=-\frac{\al-\be-1}{\al+\be-1}a_{-1,-\be-1}.
\end{equation}
Assume $\al=\be+1$. Then $a_{1,-\be+1}=0$ by (\ref{143}) and $D(L_{1})=0$, which is a contradiction. Therefore $\al\neq \be+1.$
Then $D(L_{1})\neq 0$ and we deduce that $$D(L_{\pm1})=L_{\pm1}\cdot \left(\sum\limits_{i=\be}^{-\be}c_{i}v_{i}\otimes v_{-i}\right),$$
where
$$c_{\be}=\frac{a_{-1,\be-1}}{\al-\be},\quad\, c_{\be+i}=\left(\prod\limits_{t=1}^{i}\frac{\al+\be+t-1}{t}\right)\frac{a_{-1,\be-1}}{\al-\be},\quad\, 1\leq i\leq -2\be.$$
According to {Lemma} \ref{L3.0}, we have $D\in \Inn(\W, \mathcal{F}_{\al}\ot \mathcal{F}_{\be})_{0}$.

{\bf Subcase 1.2:} $p_{-1}=\be-1$ and $q_{-1}=l$. Then
$$D(L_{1})=a_{1,l+1}v_{l+1}\otimes\,v_{-l}\neq 0, \quad\quad\, D(L_{-1})=\sum\limits_{i=\be-1}^{l}a_{-1,i}v_{i}\otimes\,v_{-i-1}\neq 0.$$
According to (\ref{1.13f1})-(\ref{1.13f3}), we have
\begin{eqnarray}
& &(l+1-\al)a_{1,l+1}=(l-1+\al)a_{-1,l-1}-(l+1-\be)a_{-1,l},   \label{181}\\
& &(l+\be)a_{1,l+1}=-(l+\al)a_{-1,l},\label{182}\\
& & (i-1+\al)a_{-1,i-1}=(i+1-\be)a_{-1,i}, \quad\quad\,i\neq l,l+1. \label{183}
\end{eqnarray}
Assume $\be=-l$, by (\ref{182}), we have $(l+\al)a_{-1,l}=0$. Since $a_{-1,l}=a_{-1,q_{-1}}\neq 0$, we have $\al=-l=\be$, which is a contradiction.
{Therefore} $\be\neq -l.$ By (\ref{182}), we have
$$a_{1,l+1}=-\frac{l+\al}{l+\be}a_{-1,l}, \quad\, D(L_{1})=-\frac{l+\al}{l+\be}a_{-1,l}v_{l+1}\otimes\,v_{-l}.$$
Since $a_{1,l+1}\neq 0$, we have $\al\neq -l$. Hence we obtain
\begin{equation}\label{184}
\frac{(\al-\be)(\al+\be-1)}{l+\be}a_{-1,l}=(l-1+\al)a_{-1,l-1}.
\end{equation}
Let $i=\be+t(-1\leq t\leq l-\be-1)$ in (\ref{183}). Then
\begin{equation}\label{185}
(t+1)a_{-1,\be+t}=(\al+\be+t-1)a_{-1,\be+t-1}, \quad\quad\,  -1\leq t\leq l-\be-1.
\end{equation}

{\bf (1.2.1)} $\al+\be\neq 1$. Then $(l-1+\al)a_{-1,l-1}\neq 0$ via (\ref{184}).

If $\be=l$, we have $\be\neq 0$ and
$$D(L_{1})=a_{1,\be+1}v_{\be+1}\otimes\,v_{-\be},\quad D(L_{-1})=a_{-1,\be-1}v_{\be-1}\otimes\,v_{-\be}+a_{-1,\be}v_{\be}\otimes\,v_{-\be-l},$$
where
$$a_{1,\be+1}=-\frac{\al+\be}{2\be}a_{-1,\be}, \quad  a_{-1,\be-1}=\frac{\al-\be}{2\be}a_{-1,\be}. $$
Obviously,
$$
D(L_{\pm 1})=L_{\pm 1}\cdot\left(\frac{a_{-1,\be}}{2\be}v_{\be}\otimes\,v_{-\be}\right).$$
Using {Lemma } \ref{L3.0},{ we know when $\be=2$ we have $D\in \Inn(\W, \mathcal{F}_{\al}\ot \mathcal{F}_{\be})_{0}\oplus\bC\d_{0, 2}.$  Otherwise,  $D\in \Inn(\W, \mathcal{F}_{\al}\ot \mathcal{F}_{\be})_{0}.$}

If $\be<l$, we have $$a_{-1,l}=\frac{(l+\be)(l-1+\al)}{(\al-\be)(\al+\be-1)}a_{-1,l-1}.$$
Using (\ref{185}), we get
\begin{equation*}
a_{-1,\be+i}=\frac{\al+\be+i-1}{i+1}\cdots\frac{\al+\be}{2}\cdot\frac{\al+\be-1}{1}\cdot{a_{-1,\be-1}}, \quad\quad\, 0\leq i\leq l-1-\be.
\end{equation*}
We check that $D(L_{\pm 1})=L_{\pm 1}\cdot\left(\sum\limits_{i=\be}^{l}c_{i}v_{i}\otimes\,v_{-i}\right)$, where
$$c_{\be}=\frac{a_{-1,\be-1}}{\al-\be}, \quad\, c_{\be+i}=\frac{\al+\be+i-1}{i}\cdots\cdot\frac{\al+\be+1}{2}\cdot\frac{\al+\be}{1}\cdot\frac{a_{-1,\be-1}}{\al-\be},\quad\, 0\leq i\leq l-\be.$$
Using the result of {Lemma } \ref{L3.0}, we have $D\in \Inn(\W, \mathcal{F}_{\al}\ot \mathcal{F}_{\be})_{0}$.

{\bf (1.2.2)}  $\al+\be=1$. Since $\be\leq l$, we have $l+1-\be\neq 0$. Then
$$D(L_{1})=a_{1,l+1}v_{l+1}\otimes\,v_{-l}=L_{1}\cdot\left(-\frac{a_{1,l+1}}{l+1-\be} \sum\limits_{i=\be}^{l} v_{i}\otimes\,v_{-i}\right).$$
Replacing $D$ by $D-D_{inn}$, where
$$D_{inn}(L_{n})=L_{n}\cdot\left(-\frac{a_{1,l+1}}{l+1-\be} \sum\limits_{i=\be}^{l} v_{i}\otimes\,v_{-i}\right)$$ for all $n\in\bZ$, we have $D(L_{1})=0$. From {Lemma } \ref{L3.0} and {Lemma } \ref{L3.02}, we obtain
$$D\in \Inn(\W, \mathcal{F}_{\al}\ot \mathcal{F}_{\be})_{0}\oplus\delta_{\al,0}\delta_{\be,1}\bC\d^{1}_{0,1}\oplus\delta_{\al,1}\delta_{\be,0}\bC\d^{1}_{1, 0}.
$$

{\bf Subcase 1.3:} $p_{-1}=\be-1$ and $q_{-1}=-\al$. Then  $\al+\be\leq 1$,
$$D(L_{1})=a_{1,l+1}v_{l+1}\otimes\,v_{-l}, \quad\quad\, D(L_{-1})=\sum\limits_{i=\be-1}^{-\al}a_{-1,i}v_{i}\otimes\,v_{-i-1}.$$
Notice that $-\al \neq l-1,l$. Then $-\al<l-1$ or $-\al\geq l+1$.

{\bf (1.3.1)} If  $-\al<l-1$, we have $(l+\be)a_{1,l+1}=0$ and $(l+1-\al)a_{1,l+1}=0$ by (\ref{1.13f1}) and (\ref{1.13f2}). Since $a_{1,l+1}\neq 0$, we have
$l=-\be, \al=1-\be.$
Since $\be\leq l$, we have $\be\leq 0.$  So
$D(L_{1})=a_{1,-\be+1}v_{-\be+1}\otimes\,v_{\be} $ and $D(L_{-1})=a_{-1,\be-1}v_{\be-1}\otimes\,v_{-\be}.$
Then
\begin{eqnarray*}
 &&D(L_{-1})=L_{-1}\cdot\left(\frac{a_{-1,\be-1}}{1-2\be}\sum\limits_{i=\be}^{-\be}v_{i}\otimes\,v_{-i}\right),\\
  && D(L_{1})=L_{1}\cdot\left(\frac{a_{-1,\be-1}}{1-2\be}\sum\limits_{i=\be}^{-\be}v_{i}\otimes\,v_{-i}\right)+(a_{-1,\be-1}+a_{1,-\be+1})v_{-\be+1}\otimes\,v_{\be}.
\end{eqnarray*}
Replacing $D$ by $D-D_{inn}$, where
$$D_{inn}(L_{n})=L_{n}\cdot \left(\frac{a_{-1,\be-1}}{1-2\be}\sum\limits_{i=\be}^{-\be}v_{i}\otimes\,v_{-i}\right)$$
for all $n\in\bZ$, we have $D(L_{-1})=0$. By {Lemma} \ref{L3.0} and {Lemma} \ref{L3.01}, we  get
$$D\in \Inn(\W, \mathcal{F}_{\al}\ot \mathcal{F}_{\be})_{0}\oplus\bC \d_{1-\be, \be}\oplus\delta_{\al,1}\delta_{\be,0}\bC\d^{1}_{1, 0}.
$$

{\bf (1.3.2)} If  $-\al\geq l+1$, we have $\al+\be\leq -1$ by $\be\leq l$.
Assume $\al+\be=-k$, where $k\geq 1$.  Then $-\al=\be+k\geq l+1$. So $l+1-k\leq \be\leq l.$
For $\be=l+t-k$, where $1\leq t\leq k$, we have $\al=-l-t$ and
$D(L_{1})=L_{1}\cdot\left(\sum\limits_{i=1}^{t}c_{i}v_{l+i}\otimes\,v_{-l-i}\right),$
where $$c_{1}=\frac{a_{1,l+1}}{k+1-t},\quad\, c_{i}=\left(\prod\limits_{j=2}^{i}\frac{t+1-j}{t-k-j}\right)\frac{a_{1,l+1}}{k+1-t}, \quad\quad\, 2\leq i\leq t.$$
Replace $D$ by $D-D_{inn}$, where $D_{inn}(L_{n})=L_{n}\cdot \left(\sum\limits_{i=1}^{t}c_{i}v_{l+i}\otimes\,v_{-l-i}\right)$ for all $n\in\bZ$. Then
$D(L_{1})=0$. Using the results of {Lemma} \ref{L3.0} and {Lemma} \ref{L3.02}, we get $D\in \Inn(\W, \mathcal{F}_{\al}\ot \mathcal{F}_{\be})_{0}$.

{\bf Subcase 1.4:} $p_{-1}=l$ and $q_{-1}=l$. Then $D(L_{1})=a_{1,l+1}v_{l+1}\otimes\,v_{-l}$ and $D(L_{-1})=a_{-1,l}v_{l}\otimes\,v_{-l-1}.$
By (\ref{1.13f1})-(\ref{1.13f2}), we have
\begin{equation}\label{198}
(l+1-\al)a_{1,l+1}=-(l+1-\be)a_{-1,l}, \quad\quad\, (l+\be)a_{1,l+1}=-(l+\al)a_{-1,l},
\end{equation}
\begin{equation*}
(\al+\be-1)(a_{1,l+1}+a_{-1,l})=0.
\end{equation*}
Assume $\al+\be\neq 1$. Then $a_{-1,l}=-a_{1,l+1}\neq 0$. By (\ref{198}), we have $\al=\be$,which is a contradiction. Hence $\al+\be=1.$  So $\be\neq -l, l+1$ by (\ref{198}) and
$a_{1,l+1}=-\frac{l+1-\be}{l+\be}a_{-1,l}.$
Then
$$D(L_{1})=L_{1}\cdot\left(\sum\limits_{i=\be}^{l}\frac{a_{-1,l}}{l+\be}v_{i}\otimes\,v_{-i}\right), \quad\,
D(L_{-1})=L_{-1}\cdot\left(\sum\limits_{i=\be}^{l}\frac{a_{-1,l}}{l+\be}v_{i}\otimes\,v_{-i}\right)-\frac{a_{-1,l}}{l+\be}v_{\be-1}\otimes\,v_{-\be}.$$
Replacing $D$ by $D-D_{inn}$, where
$$D_{inn}(L_{n})=L_{n}\cdot \left(\sum\limits_{i=\be}^{l}\frac{a_{-1,l}}{l+\be}v_{i}\otimes\,v_{-i}\right)$$ for all $n\in\bZ$, we have
$D(L_{1})=0$ and $D(L_{-1})\neq 0$ . Using {Lemma} \ref{L3.02}, we  get
$$
D\in \Inn(\W, \mathcal{F}_{\al}\ot \mathcal{F}_{\be})_{0}\oplus\bC \d_{1-\be, \be}\oplus\delta_{\al,1}\delta_{\be,0}\bC\d^{1}_{1, 0}.
$$

{\bf Subcase 1.5:} $p_{-1}=l$ and $q_{-1}=-\al>l$.
Then $\al<-l$ and
$$D(L_{1})=a_{1,l+1}v_{l+1}\otimes\,v_{-l}, \quad\quad\, D(L_{-1})=\sum\limits_{i=l}^{-\al}a_{-1,i}v_{i}\otimes\,v_{-i-1}.$$
By (\ref{1.13f1})-(\ref{1.13f2}), we have
\begin{equation}\label{199}
(l+1-\al)a_{1,l+1}=-(l+1-\be)a_{-1,l}, \quad\quad\, (l+\be)a_{1,l+1}=-(l+\al)a_{-1,l}+(l+2-\be)a_{-1,l+1},
\end{equation}
\begin{equation}\label{200}
(\al+\be-1)(a_{1,l+1}+a_{-1,l})=(l+2-\be)a_{-1,l+1}.
\end{equation}
Assume $l+1-\al=0$. Since $a_{1,l+1}\neq 0, a_{-1,l}\neq 0$, we have $l+1-\be=0$ via (\ref{199}). Hence $\al=\be$, a contradiction.
Therefore, $l+1-\al\neq 0$. Moreover, $l+1-\be\neq 0$ and
$$a_{1,l+1}=-\frac{l+1-\be}{l+1-\al}a_{-1,l}\neq -a_{-1,l}, \quad\, a_{1,l+1}+a_{-1,l}=\frac{\be-\al}{l+1-\al}a_{-1,l}\neq 0.$$

{\bf (1.5.1)} If $\al+\be\neq 1$, by (\ref{200}), we have $l+2-\be\neq 0$ and
$$a_{-1,l+1}=\frac{\al+\be-1}{l+2-\be}\cdot \frac{\be-\al}{l+1-\al}a_{-1,l}\neq 0.$$

{When} $-\al=l+1$, we have
$a_{1,l+1}=-\frac{\al+\be}{2\al}a_{-1,l}$, $a_{-1,l+1}=-\frac{\al-\be}{2\al}a_{-1,l}$
and
$$D(L_{1})=-\frac{a_{-1,l}}{2\al}\Big((\al+\be)v_{-\al}\otimes\,v_{\al+1}\Big), \,\, D(L_{-1})=-\frac{a_{-1,l}}{2\al}\Big(-2\al v_{-\al-1}\otimes\,v_{\al}+(\al-\be)v_{-\al}\otimes\,v_{\al-l}\Big).$$
Then we check that $D(L_{\pm 1})=L_{\pm 1}\cdot\left(\frac{a_{-1,l}}{2\al} v_{-\al}\otimes\,v_{\al}\right)$.
Replace $D$ by $D-D_{inn}$, where $D_{inn}(L_{n})=L_{n}\cdot \left(\frac{a_{-1,l}}{2\al} v_{-\al}\otimes\,v_{\al}\right)$ for all $n\in\bZ$. Then
$D(L_{\pm 1})=0$. Using the result of {Lemma} \ref{L3.0} and noticing $\be\leq l$, we get
$D\in \Inn(\W, \mathcal{F}_{\al}\ot \mathcal{F}_{\be})_{0}.$

{When} $-\al=l+k$, where $k\geq 2$, we have
$$a_{1,l+1}=-\frac{\al+\be+k-1}{2\al+k-1}a_{-1,l}, \quad\quad\,a_{-1,l+1}=-\frac{\al+\be-1}{\al+\be+k-2}\cdot\frac{\al-\be}{2\al+k-1}a_{-1,l}.$$
Moreover, by (\ref{1.13f3}), we have
$$a_{-1,l+i}=-\left(\prod\limits_{t=k+1-i}^{k-1}\frac{t}{\al+\be+t-2}\right)\frac{\al+\be-1}{\al+\be+k-2}\cdot\frac{\al-\be}{2\al+k-1}a_{-1,l}, \quad\quad\, 2\leq i\leq k.$$
Then
$$D(L_{\pm 1})= L_{\pm 1}\cdot\left(\sum\limits_{i=0}^{k-1}c_{i}v_{-\al-i}\otimes\,v_{\al+i}\right),$$
where
$$c_{k-1}=\frac{1}{2\al+k-1}a_{-1,l}, \quad\, c_{i-1}=\frac{i}{\al+\be+i-1}c_{i}, \quad\, i=1, \cdots, k-1.$$
Replacing $D$ by $D-D_{inn}$, where $D_{inn}(L_{n})=L_{n}\cdot\left (\sum\limits_{i=0}^{k-1}c_{i}v_{-\al-i}\otimes\,v_{\al+i}\right)$ for all $n\in\bZ$,  we have
$D(L_{\pm 1})=0$.By {Lemma} \ref{L3.0} and  $\be\leq l$, we get
$D\in \Inn(\W, \mathcal{F}_{\al}\ot \mathcal{F}_{\be})_{0}.$

{\bf (1.5.2)} If $\al+\be=1$, we have $(l+2-\be)a_{-1,l+1}=0$ by (\ref{200}).

{When} $-\al=l+1$, we have $a_{-1,l+1}\neq 0$. So $l=\be-2$, $\be\neq 1$ and  $a_{1,l+1}=\frac{1}{2\be-2}a_{-1,l}.$
Then
$$D(L_{1})=L_{1}\cdot\left(-\frac{a_{-1,l}}{2\be-2} v_{\be-1}\otimes\,v_{-\be+1}\right),$$
$$D(L_{-1})=L_{-1}\cdot\left(-\frac{a_{-1,l}}{2\be-2} v_{\be-1}\otimes\,v_{-\be+1}\right)+\left(a_{-1, l+1}+\frac{2\be-1}{2\be-2}a_{-1,l}\right)v_{\be-1}\otimes\,v_{-\be}.$$
Replace $D$ by $D-D_{inn}$, where $D_{inn}(L_{n})=L_{n}\cdot\left(-\frac{a_{-1,l}}{2\be-2} v_{\be-1}\otimes\,v_{-\be+1}\right)$ for all $n\in\bZ$. Then
$D(L_{1})=0$. By {Lemma} \ref{L3.0} and {Lemma} \ref{L3.02} and note $\be\leq l$, we get
$D\in \Inn(\W, \mathcal{F}_{\al}\ot \mathcal{F}_{\be})_{0}.$

{When} $-\al=l+k$, where $k\geq 2$, we have $\be=1-\al=l+k+1$ and $l+2-\be=1-k\leq -1$. So $a_{-1,l+1}=0$ by (\ref{200}).
Then we check that $$D(L_{1})= L_{1}\cdot\left(-\sum\limits_{i=l+1}^{l+k}\frac{a_{1,l+1}}{k}v_{i}\otimes\,v_{-i}\right).$$
Replace $D$ by $D-D_{inn}$, where $D_{inn}(L_{n})=L_{n}\cdot\left(-\sum\limits_{i=l+1}^{l+k}\frac{a_{1,l+1}}{k}v_{i}\otimes\,v_{-i}\right)$ for all $n\in\bZ$. Then
$D(L_{1})=0$. From {Lemma} \ref{L3.0}, {Lemma} \ref{L3.02} and $\be\leq l$, we get
$D\in \Inn(\W, \mathcal{F}_{\al}\ot \mathcal{F}_{\be})_{0}.$

{\bf Subcase 1.6:} $p_{-1}=l+1$ and $q_{-1}=-\al\geq l+1$.
By (\ref{1.13f1})-(\ref{1.13f2}), we have
\begin{equation}\label{201}
(l+1-\al)a_{1,l+1}=0, \quad\, (l+\be)a_{1,l+1}=(l+2-\be)a_{-1,l+1}.
\end{equation}
Since $a_{1,l+1}\neq 0$, we have $\al=l+1$. Note that $-\al\geq l+1$. Then $l\leq -1$ and $\al\leq 0$.
{ So $\be\leq l\leq -1$ and $\al+\be\leq -1$.}
Hence $D(L_{1})=a_{1,\al}v_{\al}\otimes\,v_{-\al+1} $ and $ D(L_{-1})=\sum\limits_{i=\al}^{-\al}a_{-1,i}v_{i}\otimes\,v_{-i-1}.$
By (\ref{201}), we have
\begin{equation}\label{202}
 (\al+\be-1)a_{1,\al}=(\al+1-\be)a_{-1,\al}.
\end{equation}
Since $\al+\be\leq -1$  and $a_{1,\al}\neq 0$,  we have $\be\neq \al+1$ according to (\ref{202}).

{\bf (1.6.1):} $\al=0$. Since $\be\leq -1$, we have $a_{-1,0}=-a_{1,0}$ via (\ref{202}). So
$$D(L_{1})=a_{1,0}v_{0}\otimes\,v_{1}, \quad\quad\,  D(L_{-1})=-a_{1,0}v_{0}\otimes\,v_{-1}.$$
Then
$D(L_{\pm 1})= L_{\pm 1}\cdot\left(-\frac{a_{1,0}}{\be}v_{0}\otimes\,v_{0}\right).$
Replacing $D$ by $D-D_{inn}$, where $D_{inn}(L_{n})=L_{n}\cdot\left(-\frac{a_{1,0}}{\be}v_{0}\otimes\,v_{0}\right)$ for all $n\in\bZ$, we have
$D(L_{\pm 1})=0$. By {Lemma} \ref{L3.0}, we get
$D\in \Inn(\W, \mathcal{F}_{\al}\ot \mathcal{F}_{\be})_{0}.$

{\bf (1.6.2):} $\al<0$. By (\ref{1.13f3}), we have
$$(i+\al-1)a_{-1, i-1}=(i+1-\be)a_{-1,i}, \quad\quad\, i=\al+1, \cdots, -\al.$$
Then
$$a_{-1,-\al-k}=\left(\prod\limits_{t=1}^{k}\frac{\al+\be+t-2}{t}\right)a_{-1,-\al},\quad\quad\, 1\leq k\leq -2\al. $$
Since $a_{-1, -\al}\neq 0$, we have $\al+\be+i\neq 0$ for $-1\leq i\leq -2\al-2$.
In addition, by (\ref{202}), we have $a_{1,\al}=\frac{\al+1-\be}{\al+\be-1}a_{-1,\al}$.
Then we check that
$D(L_{\pm 1})=L_{\pm 1}\cdot \left(\sum\limits_{i=\al}^{-\al}c_{i}v_{i}\otimes\,v_{-i}\right),$
where
$$c_{-\al}=-\frac{a_{-1, -\al}}{\al-\be},\quad\quad\, c_{-\al-k}=\left(\prod\limits_{t=1}^{k}\frac{\al+\be+t-1}{t}\right)c_{-\al},\quad\quad\, 1\leq k\leq -2\al.$$
Replacing $D$ by $D-D_{inn}$, where $D_{inn}(L_{n})=L_{n}\cdot \left(\sum\limits_{i=\al}^{-\al}c_{i}v_{i}\otimes\,v_{-i}\right)$ for all $n\in\bZ$,  we have
$D(L_{\pm 1})=0$. By {Lemma} \ref{L3.0}, we get
$D\in \Inn(\W, \mathcal{F}_{\al}\ot \mathcal{F}_{\be})_{0}$.

\bigskip

{\bf Case 2:} $\al=\be$.
By (\ref{131}), we have
\begin{equation}\label{203}
 (i+1-\be)(a_{1,i+1}+a_{-1,i})=(i-1+\be)(a_{1,i}+a_{-1,i-1}).
\end{equation}
Since $\{a_{-1,i}\neq 0\mid i\in\bZ\}$ is finite, we deduce that $a_{-1, i}=0$ for $i\leq \be-2$ via (\ref{137}).

{\bf Subcase 2.1:}  $\al=\be>0$. Then $l>0$ and
$$D(L_1)=a_{1,\be}v_{\be}\otimes\,v_{-\be+1}+a_{1,l+1}v_{l+1}\otimes\,v_{-l}
=a_{1,\be}v_{\be}\otimes\,v_{-\be+1}+L_{1}\cdot\left(\sum\limits_{i=\be}^{l}c_{i}v_{i}\otimes\,v_{-i}\right),$$
where
$$c_{l}=-\frac{a_{1,l+1}}{l+\be}, \quad\, (i-\be)c_{i}=(i-1+\be)c_{i-1},\quad\, i=\be+1, \cdots, l.$$
Replacing $D$ by $D-D_{inn}$, where $D_{inn}(L_{n})=L_{n}\cdot \left(\sum\limits_{i=\be}^{l}c_{i}v_{i}\otimes\,v_{-i}\right)$ for all $n\in\bZ$,
we have $D(L_{1})=a_{1,\be}v_{\be}\otimes\,v_{-\be+1}.$
By (\ref{203}), we have
\begin{equation*}
a_{-1,\be}=(2\be-1)(a_{1,\be}+a_{-1,\be-1}), \quad\, (i+1-\be)a_{-1,i}=(i-1+\be)a_{-1,i-1}, \quad\, i\neq \be-1,\be.
\end{equation*}
Since $\{a_{-1,i}\neq 0\mid i\in\bZ\}$ is finite, we deduce that $a_{-1, i}=0$ for $i\leq \be-2$ or $i\geq \be$. Furthermore,
$(2\be-1)(a_{1,\be}+a_{-1,\be-1})=0$. So $a_{-1,\be-1}=-a_{1,\be}$ and $D(L_{-1})=-a_{1,\be}v_{\be-1}\otimes\,v_{-\be}.$
By (\ref{1.6c}), we have
\begin{equation}\label{204}
a_{2,\be-1}=\frac{\be}{2\be-3}a_{1,\be},
\end{equation}
\begin{equation}\label{205}
(i+1-\be)a_{2,i+1}=(i-2+\be)a_{2,i},\quad\, i\neq \be-1,\be,\be+1.
\end{equation}
Since $\{a_{2,i}\neq 0\mid i\in\bZ\}$ is finite, we deduce that $a_{2, i}=0$ for $i\leq \be-1$ via (\ref{205}).
Using (\ref{204}), we have $a_{1,\be}=0$. Therefore, $D(L_{\pm 1})=0$. By {Lemma} \ref{L3.0}, we have
$$D\in \Inn(\W, \mathcal{F}_{\al}\ot \mathcal{F}_{\be})_{0}\oplus\delta_{\al,1}\delta_{\be,1}\bC\d_{1,1}.$$

{\bf Subcase 2.2:}  $\al=\be=0$. Then $l\geq0$ and
$D(L_1)=a_{1,0}v_{0}\otimes\,v_{1}+a_{1,l+1}v_{l+1}\otimes\,v_{-l}.$

If $l=0$, we have
$D(L_1)=a_{1,0}v_{0}\otimes\,v_{1}+a_{1,1}v_{1}\otimes\,v_{0}.$
If $l>0$, we have
$$D(L_1)=a_{1,0}v_{0}\otimes\,v_{1}+a_{1,l+1}v_{l+1}\otimes\,v_{-l}
=a_{1,0}v_{0}\otimes\,v_{1}+a_{1,l+1}v_{1}\otimes\,v_{0}+L_{1}\cdot\left(\sum\limits_{i=1}^{l}-\frac{a_{1,l+1}}{i}v_{i}\otimes\,v_{-i}\right).$$
Then we assume that $D(L_{1})=a_{1,0}v_{0}\otimes\,v_{1}+a_{1,1}v_{1}\otimes\,v_{0}$. Using (\ref{131}), we obtain
\begin{equation}\label{206}
a_{1,0}+a_{1,1}=-a_{-1,-1}-a_{-1,0},
\end{equation}
$$a_{-1,-2}=a_{-1, 1}=0,\quad\, (i+1)a_{-1,i}=(i-1)a_{-1,i-1}, \quad\, i\neq 0,\pm1.$$
It is easy to deduce that $a_{-1,i}=0$ for $i\neq 0,{-1}$. Hence
$D(L_{-1})=a_{-1,-1}v_{-1}\otimes\,v_{0}+a_{-1,0}v_{0}\otimes\,v_{-1}.$
By (\ref{1.6c}), we have
\begin{equation}\label{207}
 2a_{2,0}+a_{2,1}=3a_{1,0}-a_{-1,0},\quad\quad\,  a_{2,1}+2a_{2,2}=3a_{1,1}-a_{-1,-1},
\end{equation}
$$(i+1)a_{2,i+1}=(i-2)a_{2,i},\quad\quad\, i\neq 0,1,2.$$
Then  $a_{2,i}=0$ for $i\neq 0,1,2$. So
$D(L_{2})=a_{2,0}v_{0}\otimes\,v_{2}+a_{2,1}v_{1}\otimes\,v_{1}+a_{2,2}v_{2}\otimes\,v_{0}.$
By (\ref{1.6cc}), we have
\begin{equation}\label{208}
2a_{-2,0}+a_{-2,-1}=3a_{-1,0}-a_{1,0},\quad\quad\,  a_{-2,-1}+2a_{-2,-2}=3a_{-1,-1}-a_{1,1},
\end{equation}
$$ (i-1)a_{-2,i-1}=(i+2)a_{-2,i}, \quad\quad\, i\neq -2,-1,0.$$
Then  $a_{-2,i}=0$ for $i\neq -2,-1,0$. So
$$D(L_{-2})=a_{-2,-2}v_{-2}\otimes\,v_{0}+a_{-2,-1}v_{-1}\otimes\,v_{-1}+a_{-2,0}v_{0}\otimes\,v_{-2}.$$
Let $m=2$ and $i=0,1$ in (\ref{1.6ccc}). Then
\begin{equation}\label{209}
a_{2,2}+a_{2,0}=-(a_{-2,-2}+a_{-2,0}), \quad\quad\, a_{-2,-1}=-a_{2,1}.
\end{equation}
Using (\ref{206})-(\ref{209}), we have
$$D=(2a_{1,0}-a_{2,0})\d_{0,0}^{1}+(a_{2,0}-a_{1,0})\d_{0,0}^{2}
+(2a_{1,1}-a_{2,2})\d_{0,0}^{3}+(a_{2,2}-a_{1,1})\d_{0,0}^{4}-a_{2,1}\d_{0,0}^{5}$$
acts on $L_{\pm1}, L_{\pm 2}$. Hence
$$D=(2a_{1,0}-a_{2,0})\d_{0,0}^{1}+(a_{2,0}-a_{1,0})\d_{0,0}^{2}
+(2a_{1,1}-a_{2,2})\d_{0,0}^{3}+(a_{2,2}-a_{1,1})\d_{0,0}^{4}-a_{2,1}\d_{0,0}^{5}.$$
That is,
$$D\in \Inn(\W, \mathcal{F}_{\al}\ot \mathcal{F}_{\be})_{0}\oplus\bC\d_{0,0}^{1}\oplus\bC\d_{0,0}^{2}\oplus\bC\d_{0,0}^{3}\oplus\bC\d_{0,0}^{4}\oplus\bC\d_{0,0}^{5}.$$

{\bf Subcase 2.3:}  $\al=\be<0$.

{\bf  (2.3.1):} $\be\leq l<-\be$.
If $l=\be$, we have
$$D(L_{1})=a_{1,\be}v_{\be}\otimes\,v_{-\be+1}+a_{1,l+1}v_{l+1}\otimes\,v_{-l}
=a_{1,\be}v_{\be}\otimes\,v_{-\be+1}+L_{1}\cdot\left(-\frac{a_{1,l+1}}{2\be}v_{\be}\otimes\,v_{-\be}\right).$$
If $\be< l<-\be$, we have
$$D(L_{1})=a_{1,\be}v_{\be}\otimes\,v_{-\be+1}+a_{1,l+1}v_{l+1}\otimes\,v_{-l}
=a_{1,\be}v_{\be}\otimes\,v_{-\be+1}+L_{1}\cdot\left(\sum\limits_{i=\be}^{l}c_{i}v_{i}\otimes\,v_{-i}\right),$$
where
$$c_{l}=-\frac{a_{1,l+1}}{l+\be}, \quad\, (i-\be)c_{i}=(i-1+\be)c_{i-1}, \quad\quad\, i=\be+1, \cdots, l.$$
Therefore, for $\be\leq l<-\be$, we assume that $D(L_{1})=a_{1,\be}v_{\be}\otimes\,v_{-\be+1}$.

By (\ref{131}), we have
\begin{equation}\label{210}
a_{-1,\be-2}=0, \quad\quad\, a_{-1,\be}=(2\be-1)(a_{1,\be}+a_{-1,\be-1}),
\end{equation}
$$(i-1+\be)a_{-1,i-1}=(i+1-\be)a_{-1,i}, \quad\quad\, i\neq \be-1,\be.$$
Then $a_{-1,i}=0$ for $i\leq \be-2$ or $i\geq -\be+1$ and
\begin{equation}\label{211}
a_{-1,\be+i}=\left(\prod\limits_{t=0}^{i-1}\frac{2\be+t}{t+2}\right)a_{-1,\be},\quad\quad\, 1\leq i\leq -2\be.
\end{equation}
By (\ref{1.6c}), we have
\begin{equation}\label{212}
 (i-2+2\be)a_{-1,i-2}-(i+1-2\be)a_{-1,i}+3a_{1,i}=(i+1-\be)a_{2,i+1}-(i-2+\be)a_{2,i}.
\end{equation}
Then we obtain
\begin{equation}\label{213}
 \be a_{-1,\be-1}=(3-2\be)a_{2,\be-1},\quad\, \be a_{-1,-\be}=(3-2\be)a_{2,-\be+3},
\end{equation}
\begin{equation*}
 (i+1-\be)a_{2,i+1}=(i-2+\be)a_{2,i},\quad\quad\, i\leq \be-1\quad\, \mbox{or}\quad\, i\geq -\be+3.
\end{equation*}
Since $\{a_{2,i}\neq 0\mid i\in\bZ\}$ is finite, we deduce that $a_{2, i}=0$ for $i\leq \be-1$ or $i\geq -\be+3$.
Using (\ref{213}), we have $$a_{-1,\be-1}=a_{-1,-\be}=0.$$
By (\ref{211}), $a_{-1,-\be}=\left(\prod\limits_{t=0}^{-2\be-1}\frac{2\be+t}{t+2}\right)a_{-1,\be}$. Since $\be<0$, we have $\prod\limits_{t=0}^{-2\be-1}\frac{2\be+t}{t+2}\neq 0$. Hence $a_{-1,\be}=0$ and $a_{-1,i}=0$ for $\be+1\leq i\leq -\be$ via (\ref{211}).
By (\ref{210}), we have $a_{1,\be}=0$.
Therefore, $D(L_{\pm1})=0.$
 According to the result of {Lemma} \ref{L3.0}, we have $$D\in \Inn(\W, \mathcal{F}_{\al}\ot \mathcal{F}_{\be})_{0}.$$

{\bf  (2.3.2):} $l\geq -\be$.

If $l=-\be$, we have
$D(L_{1})=a_{1,\be}v_{\be}\otimes\,v_{-\be+1}+a_{1,-\be+1}v_{-\be+1}\otimes\,v_{\be}.$
If $l>-\be$, we have
$$D(L_{1})=a_{1,\be}v_{\be}\otimes\,v_{-\be+1}+{(2\be-1)}x_{-\be+1}v_{-\be+1}\otimes\,v_{\be}+L_{1}\cdot\left(\sum\limits_{i=-\be+1}^{l}x_{i}v_{i}\otimes\,v_{-i}\right),$$
where
$$x_{l}=-\frac{a_{1,l+1}}{l+\be}, \quad\, (i-\be)x_{i}=(i-1+\be)x_{i-1}, \quad\quad\, i=-\be+2, \cdots, l.$$
Therefore,  for $l\geq -\be$, we assume that
$D(L_{1})=a_{1,\be}v_{\be}\otimes\,v_{-\be+1}+a_{1,-\be+1}v_{-\be+1}\otimes\,v_{\be}$.

By (\ref{131}), we have
\begin{equation}\label{214}
 (1-2\be)a_{1,\be}=(2\be-1)a_{-1,\be-1}-a_{-1,\be}, \quad\, (2\be-1)a_{1,-\be+1}=a_{-1,-\be-1}-(2\be-1){a_{-1,-\be}},
\end{equation}
$$a_{-1,\be-2}=a_{-1,-\be+1}=0, \quad\,(i-1+\be)a_{-1,i-1}=(i+1-\be)a_{-1,i}, \quad\quad\, i\leq \be-2\quad\, \mbox{or}\quad\, i\geq -\be+2.$$
Then  $a_{-1,i}=0$ for $i\leq \be-2$ or $i\geq -\be+1$.

By (\ref{1.6c}), we have
\begin{equation}\label{215}
 \be a_{-1,\be-1}=(3-2\be)a_{2,\be-1},\quad\, \be a_{-1,-\be}=(3-2\be)a_{2,-\be+3},
\end{equation}
\begin{equation*}
 (i+1-\be)a_{2,i+1}=(i-2+\be)a_{2,i},\quad\quad\, i\leq \be-2\quad\, \mbox{or}\quad\, i\geq -\be+3.
\end{equation*}
Since $\{a_{2,i}\neq 0\mid i\in\bZ\}$ is finite, we get $a_{2, i}=0$ for $i\leq \be-1$ or $i\geq -\be+3$.
Using (\ref{215}), we have $a_{-1,\be-1}=a_{-1,-\be}=0.$
Let $i=\be-1$ in (\ref{131}). Then  $a_{1,\be}=0$.
Furthermore, using (\ref{131}), we have
\begin{equation}\label{216}
 (2\be-1)a_{1,-\be+1}=a_{-1,-\be-1},
\end{equation}
\begin{equation}\label{217}
 (i-1+\be)a_{-1,i-1}=(i+1-\be)a_{-1,i}, \quad\quad\, i\neq -\be, -\be+1.
\end{equation}
Let $i=\be,\be+1,\cdots, -\be-1$ in (\ref{217}) and use $a_{-1,\be-1}=0$, then we deduce that $a_{-1,i}=0$ for $\be\leq i\leq -\be-1$.
Hence $a_{1,-\be+1}=0$. Therefore, $D(L_{\pm 1})=0$.
According to {Lemma} \ref{L3.0}, we have $$D\in \Inn(\W, \mathcal{F}_{\al}\ot \mathcal{F}_{\be})_{0}.$$
\end{proof}

\subsection{}

By Lemma \ref{L1.3}, we assume $D(L_{0})=v_{0}\otimes\,v_{0}$ in this case.
\begin{lemm}\label{L3.4}
For $(\al,\be)=(0,0)$, if $D(L_{0})=v_{0}\otimes\,v_{0}$, we have  $$D\in\Inn(\W,\mathcal{F}_{0}\ot \mathcal{F}_{0})_{0}\oplus\bC\d_{0,0}^{1}\oplus\bC\d_{0,0}^{2}\oplus\bC\d_{0,0}^{3}\oplus\bC\d_{0,0}^{4}\oplus\bC\d_{0,0}^{5}.$$
\end{lemm}
\begin{proof}
For $(\al,\be)=(0,0)$, if $D(L_{0})=v_{0}\otimes\,v_{0}$,  replacing $D$ by $D-\d_{0,0}^{1}$, we have $D(L_{0})=0$. Using the results of Lemma \ref{L3.0}-Lemma \ref{L3.03}, we have $$D\in\Inn(\W,\mathcal{F}_{0}\ot \mathcal{F}_{0})_{0}\oplus\bC\d_{0,0}^{1}\oplus\bC\d_{0,0}^{2}\oplus\bC\d_{0,0}^{3}\oplus\bC\d_{0,0}^{4}\oplus\bC\d_{0,0}^{5}.$$
\end{proof}

By Lemma \ref{L3.0}-Lemma \ref{L3.4}
our main result in this section can be stated as follows.

\begin{prop}\label{MP2}
For $(\al,\be)\in\bZ^2$, we have
\begin{eqnarray*}
  {\rm H}^1(\W, \mathcal{F}_{\al}\ot \mathcal{F}_{\be})=
   \begin{cases}
   \oplus_{i=1}^5\bC \d_{0,0}^i,& (\al,\be)=(0,0);\\
   \bC\d_{0,1}\oplus\bC \d_{0,1}^1,& (\al,\be)=(0,1);\\
   \bC\d_{1,0}\oplus\bC \d_{1,0}^1,&(\al,\be)=(1,0);\\
   \bC\d_{0,2},& (\al,\be)=(0,2);\\
   \bC\d_{2,0},& (\al,\be)=(2,0);\\
   \bC\d_{1,1},& (\al,\be)=(1,1);\\
   \bC\d_{\al,\be},& \al+\be=1, \be\neq 0,1;\\
   0,& \text{otherwise},
   \end{cases}
\end{eqnarray*}
where  $\d_{0,0}^i (i=1,2\cdots,5), \d_{1,0}^{1}, \d_{0,1}^{1}, \d_{\al,\be}$ are defined in Proposition \ref{P1.6--}, Proposition \ref{P3.0}, and  Proposition \ref{P3.1}.
\end{prop}

By utilizing Proposition \ref{MP1} and Proposition \ref{MP2}, we present the main theorem of this paper.

\begin{theo}\label{T}
 For $\al,\be\in\bC$, we have
\begin{equation*}
\dim \mathrm{H}^1(\mathcal{W}, \mathcal{F}_{\alpha} \otimes \mathcal{F}_{\beta}) = \begin{cases}
5, & (\alpha, \beta) = (0,0); \\
2, & (\alpha, \beta) = (0,1) \text{ or } (1,0); \\
1, & (\alpha, \beta) = (0,2), (2,0), (1,1), \text{ or } \alpha + \beta = 1, \beta \neq 0,1; \\
0, & \text{otherwise}.
\end{cases}
\end{equation*}
\end{theo}

\section{Applications}\label{section 6}

In this section, we discuss some  applications of Theorem \ref{T}.

 In \cite{D1}, Drinfeld introduced the notion of Lie bialgebras as a framework for investigating the Yang-Baxter equations. A Lie bialgebra, denoted as $(L, \delta)$, consists of a Lie algebra $L$ equipped with a Lie cobracket $\delta: L \rightarrow L \wedge L$ that satisfies the co-Jacobi identity, as described in \cite{M1,M2}. The mapping $\delta$ operates as a 1-cocycle on the Lie algebra $L$, with coefficients in the $L$-module $L \wedge L$. This action of $L$ on $L \wedge L$ is carried out through the adjoint representation. When $\delta$ is a 1-coboundary with coefficients in $\wedge^2 L$, the Lie bialgebra $(L, \delta)$ is termed a coboundary Lie bialgebra.

\begin{theo}[\cite{NT,Ta}]\label{P0} Let $\W$ be the Witt algebra. Then
${\rm H}^1({\W},\W\wedge \W)=0$. In particular, all  Lie bialgebra structures on $\W$ are coboundary.
\end{theo}

\begin{proof}
  It follows from  $\W\cong \mathcal{F}_{-1}$ as $\W$-modules and Theorem \ref{T}.
  \end{proof}

For $\al\in\bC$,  the {\bf Ovsienko-Roger  algebra} $\W(\al)=\W\ltimes \mathcal{F}_{\al}$ \cite{OR2}(see also \cite{GJP}) is the semi-direct product of $\W$ and $\mathcal{F}_\al$ with the following brackets:
$$
[L_m,L_n]=(m-n)L_{m+m},\quad [L_m, v_n]=-(\al m+n)v_{m+n},\quad [v_m,v_n]=0,\quad\forall m,n\in\bZ.
$$

\begin{prop}\label{PropA}
$$
 \dim {\rm H}^1(\W(\al),\W(\al)\ot \W(\al))=\begin{cases} 7,&\al=0;\\
1,&\al=1;\\
1,&\al=\frac{1}{2};\\
0,&{otherwise}.
\end{cases}
$$
\end{prop}

\begin{proof}
Since $\mathcal{F}_\al$ is an abelian ideal of $\W(\al)$, using   the Hochschild-Serre spectral sequence \cite{We}, we have the following exact sequence:
$$
0\to \mathrm{H}^1(\W,(\W(\al)\ot \W(\al))^{\mathcal{F}_{\al}})\to \mathrm{H}^1(\W(\al),\W(\al)\ot \W(\al))\to \mathrm{H}^1(\mathcal{F}_\al,\W(\al)\ot \W(\al))^{\W(\al)}.
$$
Note that the action of $\W$ on $\mathrm{H}^1(\mathcal{F}_\al,\W(\al)\ot \W(\al))$ is induced by the action of $\W$ on
$\Der(\mathcal{F}_\al,\W(\al)\ot \W(\al)))$ defined by
$$(x\cdot d)(z)=x\cdot d(z)-d[x,z],\quad \forall x\in \W(\al),z\in \W(\al)\ot \W(\al).$$
Since
\begin{eqnarray*}(\W(\al)\ot\W(\al))^{\mathcal{F}_\al}&=&\{v\in \W(\al)\ot\W(\al)\mid x\cdot v=0,\forall x\in \mathcal{F}_\al\}=\mathcal{F}_{\al}\ot \mathcal{F}_{\al},\\
\mathrm{H}^1(\mathcal{F}_\al,\W(\al)\ot\W(\al))^{\W}&=&{\rm Hom}_{\W}(\mathcal{F}_\al,\W(\al)\ot\W(\al)),
\end{eqnarray*}
it suffices to compute
$
\mathrm{H}^1(\W,\mathcal{F}_{\al}\ot \mathcal{F}_{\al})$ and ${\rm Hom}_{\W}(\mathcal{F}_\al,\W(\al)\ot\W(\al))$.
From Theorem \ref{T}, we obtain
$$
\dim \mathrm{H}^1(\W,\mathcal{F}_{\al}\ot \mathcal{F}_{\al})=\begin{cases} 5,&\al=0;\\
1,&\al=1;\\
1,&\al=\frac{1}{2};\\
0,&{otherwise}.
\end{cases}
$$ By direct computation, we have
\begin{equation*}\label{L6-3}
{\rm Hom}_{\W}(\mathcal{F}_\al,\W(\al)\ot\W(\al))=
\begin{cases}
\mathbb{C} \tilde\d_{0}^6\oplus \mathbb{C}\tilde\d_{0}^7,&\alpha=0;\\
0,&\text{otherwise},
\end{cases}
\end{equation*}
where
\begin{equation*}\label{def-6-3}
\tilde\d_{0}^6(v_{m})=v_{0}\ot v_{m}, \quad\, \tilde\d_{0}^7(v_{m})=v_{m}\ot v_{0}, \quad \forall m\in\mathbb{Z}.
\end{equation*}
It is easy to see  $\tilde\d_{0}^6$ and $\tilde\d_{0}^7$ lies in the image of the restriction map. This completes the proof.
\end{proof}
\begin{rema}{\rm The algebra $W(0)$ is isomorphic to the twisted Heisenberg-Virasoro algebra when $\alpha=0$. In \cite{LPZ}, ${\rm H}^1(\W(-1),\W(-1)\ot \W(-1))$ was calculated; however, it is regrettable that the 1-cocycle $\d_{0,0}^5$ was omitted. When $\alpha=-1$, the algebra $\W(-1)$ is isomorphic to the BMS algebra. The calculation of ${\rm H}^1(\W(0),\W(0)\ot \W(0))$ was successfully accomplished in \cite{LSX}.}

\end{rema}
\begin{coro}
$$
{\rm H}^1(\W(\al),\W(\al)\wedge \W(\al))=\begin{cases}
 \bigoplus_{i=1}^3\bC D_0^i,&\al=0;\\
 0,&otherwise,
\end{cases}
$$
where
\begin{eqnarray*}
   & & D_{0}^{1}(L_{m})=v_{m}\otimes\,v_{0}-v_{0}\otimes\,v_{m},\   D_{0}^{2}(L_{m})=m(v_{m}\otimes\,v_{0}-v_{0}\otimes\,v_{m}),\ D_{0}^{3}(v_{m})=v_{m}\ot v_{0}-v_{0}\ot v_{m},\\
   &&D_{0}^{1}(v_m)=D_{0}^{2}(v_m)=D_{0}^{3}(L_m)=0
   \end{eqnarray*}
for $m\in\bZ$. In particular, all Lie bialgebra structures on $W(\al)$ are coboundary if and only if $\al\neq 0$.
\end{coro}

\begin{proof}
By Proposition \ref{PropA}, we only need to consider $\alpha=0,1,\frac{1}{2}$. By utilizing the explicit non-trivial 1-cocycles in Proposition \ref{P1.6--}, Proposition \ref{P3.0}, and Proposition \ref{P3.1}, one can easily determine all non-trivial 1-cocycles with coefficients in $\mathcal{W}(\alpha)\wedge\mathcal{W}(\alpha)$ through the skew symmetry property.
\end{proof}

The {\bf twisted Schr\"{o}dinger-Virasoro algebra}
$\mathcal{S}$ introduced in \cite{RU},  is a Lie
algebra with $\C$-basis $\{L_n,M_n,Y_n, \mid n \in\Z\}$ and the following Lie brackets,
\begin{eqnarray*}
&&[L_m,\,L_{n}]=(m-n)L_{m+n},\quad [L_m,\,Y_n]=\left(\frac{m}{2}-n\right)Y_{m+n},\quad [L_m,\,M_n]=-nM_{n+m},\\
&&[Y_m,\,Y_{n}]=(m-n)M_{m+n},\quad [Y_m,\,M_n]=[M_m,\,M_{n}]=0
\end{eqnarray*}
for $m,n\in\bZ$. Let $\mathcal{I}={\rm span}_{\bC}\{M_n,Y_n\mid n\in\bZ\}$. It is clear that $\mathcal I$ is an ideal of $\mathcal S$ and $\mathcal I\cong \mathcal{F}_{0}\oplus \mathcal{F}_{-\frac{1}{2}}$ as $\W$-modules.

\begin{prop}
$\dim {\rm H}^1(\mathcal S,\mathcal S\ot \mathcal S )=7$.
\end{prop}
\begin{proof}
For $\mathcal I$ is a  ideal of $\mathcal S$,  we have the following exact sequence:
$$
0\to \mathrm{H}^1(\W,(\mathcal S\ot \mathcal S)^{\mathcal I})\to \mathrm{H}^1(\mathcal S,\mathcal S\ot \mathcal S)\to \mathrm{H}^1(\mathcal I,\mathcal S\ot \mathcal S)^{\W}.
$$
Since
$
(\mathcal S\ot\mathcal S)^{\mathcal I}=\{v\in \mathcal S\ot\mathcal S\mid x\cdot v=0,\forall x\in \mathcal I\}=\mathcal{F}_{0}\ot \mathcal{F}_{0}.
$
From Theorem \ref{T} it follows that the dimension of the first space in the exact sequence is $5$.
 By direct computation, we have
$
\mathrm{H}^1(\mathcal I,\mathcal S\ot \mathcal S)^{\W}=\mathbb{C} \tilde\d_6\oplus \mathbb{C}\tilde\d_7,
$
where
\begin{eqnarray*}\label{def-6-3}
&& \tilde\d_{6}(M_{m})=M_{0}\ot M_{m},\quad \tilde\d_{6}(Y_m)=M_0\ot Y_m,\quad \tilde\d_{7}(M_{m})=M_{m}\ot M_{0},\quad \tilde\d_{7}(Y_m)=Y_m\ot M_0.
  \end{eqnarray*}
It is easy to see that $\tilde\d_{0}^6$ and $\tilde\d_{0}^7$ lie in the  image of the restriction map. This completes the proof.
\end{proof}

\begin{coro}
${\rm H}^1(\mathcal S,\mathcal S\wedge \mathcal S )=\bigoplus_{i=1}^3\bC D_i$, where
\begin{eqnarray*}
   && D_{1}(L_{m})=M_{0}\otimes M_{m}-M_{m}\otimes M_{0},\quad D_{2}(L_{m})=mM_{0}\otimes M_{m}-mM_{m}\otimes M_{0},\\
   && D_{3}(M_{m})=M_{0}\ot M_{m}-M_{m}\ot M_{0},\quad  D_{3}(Y_m)=M_0\ot Y_m-Y_m\ot M_0,\\
   &&D_1(M_m)=D_1(Y_m)=D_2(M_m)=D_2(Y_m)=D_3(L_m)=0
   \end{eqnarray*}
   for $m\in\bZ$.
\end{coro}

	\section*{Acknowledgements}
We would like to thank Professor Claude Roger for his valuable comments and Xia Xie for her suggestions. We also greatly appreciate the input from the anonymous reviewers. This research was supported by the National Natural Science Foundation of China (12071405).

\end{document}